\documentclass[12pt,twoside]{amsart}
\usepackage{amsmath,bm}
\usepackage{amscd}
\usepackage{amsmath}
\usepackage{latexsym}
\usepackage{amsfonts}
\usepackage{amssymb}
\usepackage{amsthm}
\usepackage{graphicx}
\usepackage{verbatim}
\usepackage{mathrsfs}
\usepackage{enumerate}
\usepackage{color}
\usepackage{titlesec}
\titleformat{\section}[block]{\normalsize\bfseries}{\arabic{section}.}{1em}{ }[]
\titleformat{\subsection}[block]{\normalsize\bfseries}{\arabic{section}.\arabic{subsection}}{1em}{}[]
\titleformat{\subsubsection}[block]{\normalsize\bfseries}{\arabic{subsection}-\alph{subsubsection}}{1em}{}[]
\titleformat{\paragraph}[block]{\small\bfseries}{[\arabic{paragraph}]}{1em}{}[]

\numberwithin{equation}{section}

\newtheorem{theorem}{Theorem}[section]
\newtheorem{lemma}[theorem]{Lemma}

\newtheorem{proposition}[theorem]{Proposition}

\newtheorem{remark}[theorem]{Remark}

\newcommand{\ZZ}{\mathbb{Z}}

\newcommand{\NN}{\mathbb{N}}

\newcommand{\RR}{\mathbb{R}}

\newcommand{\h}{h^{n+1}}
\newcommand{\hn}{h^{n}}
\newcommand{\uu}{{u}^{n+1}}
\newcommand{\un}{{u}^{n}}
\newcommand{\dd}{\,\mathrm{d}}
\newcommand{\B}{\dot B^{0}_{[2, \infty],1}}
\newcommand{\Bo}{\dot B^{1}_{[2, \infty],1}}
\newcommand{\Bt}{\dot B^{2}_{[2, \infty],1}}
\newcommand{\e}{\varepsilon}

\newcommand{\rd}{\,\mathrm{d}}
\DeclareMathOperator{\Div}{div}

\makeatletter

\newcommand{\Rmnum}[1]{\expandafter\@slowromancap\romannumeral #1@}
\makeatother
\allowdisplaybreaks
\begin{document}
\title[] {Ill-posedness issue for the 2D viscous shallow water equations in some
critical Besov spaces }
\author[Q. Chen, Y. Nie]{Qionglei Chen, Yao Nie}
\address{Institute of Applied Physics and Computational
Mathematics, Beijing 100191, P.R. China}
\email{chen\_qionglei@iapcm.ac.cn}
\address{Institute of Applied Physics and Computational
Mathematics, Beijing 100191, P.R. China}
\email{nieyao930930@163.com}
\begin{abstract}
We study the Cauchy problem of the 2D viscous shallow water equations in some critical Besov spaces $\dot B^{\frac{2}{p}}_{p,1}(\RR^2)\times \dot B^{\frac{2}{p}-1}_{p,q}(\RR^2)$. As is known,  this system is locally well-posed for large initial data as well as globally well-posed for small initial data in $\dot B^{\frac{2}{p}}_{p,1}(\RR^2)\times \dot B^{\frac{2}{p}-1}_{p,1}(\RR^2)$ for $p<4$ and ill-posed in $\dot B^{\frac{2}{p}}_{p,1}(\RR^2)\times \dot B^{\frac{2}{p}-1}_{p,1}(\RR^2)$ for $p>4$. In this paper, we  prove that this system is ill-posed for the critical case $p=4$ in the sense of ``norm inflation''. Furthermore, we also show that the system is ill-posed in $\dot B^{\frac{1}{2}}_{4,1}(\RR^2)\times \dot B^{-\frac{1}{2}}_{4,q}(\RR^2)$ for  any $q\neq 2$.
\end{abstract}
\maketitle
\section{Introduction}
The 2D viscous shallow water equations read as follows:
\begin{equation}\label{sw}
\left\{ \aligned
    & \partial_t h+\Div (h {u})=0,\,\,\, &t>0, x\in \RR^2,\\
     &h(\partial_t {u}+ {u}\cdot\nabla {u})-\nu\nabla\cdot(h\nabla u)+h\nabla h=0,\,\,\quad\qquad  &t>0, x\in \RR^2,\\
     &h(0, x)=h_0(x),\, {u}(0, x)= {u}_0(x),\,\, &x\in \RR^2.
\endaligned
\right.
\end{equation}
where $h(t,x)$ denotes the height of fluid surface, $ u(t,x)=(u_1(t,x), u_2(t,x))$ represents the horizontal velocity field and $\nu>0$  is the viscous coefficient.

The well-posedness of system \eqref{sw} has been widely investigated during the past 30 years. Readers can refer to \cite{BDM} for more details. Bui \cite{Bui}  proved the local existence and uniqueness
of classical solutions to the Cauchy-Dirichlet problem for system \eqref{sw}
with initial data in $C^{2+\alpha}$.  Kloeden~\cite{Kl85} and Sundbye \cite{Su96} independently showed global existence and uniqueness of classical solutions to the Cauchy-Dirichlet problem in Sobolev spaces for small initial data. Sundbye \cite{Su98} established a global existence and uniqueness theorem of strong solutions for the Cauchy problem for equations \eqref{sw} with small initial data. Wang and Xu \cite{WX05}  got the local existence of solution for all size initial data and global existence for small initial data $u_0$ if $h_0-\bar{h}_0$ is small enough in $H^{2+s}$ for any $s>0$. Chen, Miao and Zhang \cite{CMZ08} introduced some kind of weighted Besov space to prove the existence and uniqueness of the solutions to a more general diffusion system  with low regularity assumptions on the initial data  as well as the initial height bounded away from zero.  For more results on well-posedness of system \eqref{sw} in Besov spaces,  readers can refer to \cite{LY16, LiY16, LY15}.

For convenience, we take $\bar{h}_0=1$, $\nu=1$, and  substitute $h$ by $1+h$ in equations \eqref{sw}, then it yields that
\begin{equation}\label{sw1}
\left\{ \aligned
    & \partial_t h+\Div  {u}+ {u}\cdot\nabla h=-h\Div  {u},\,\,\, &t>0, x\in \RR^2,\\
     &\partial_t {u}+ {u}\cdot\nabla {u}-\Delta {u}+\nabla h=\nabla(\ln(1+h))\cdot\nabla {u},\,\,\quad\qquad  &t>0, x\in \RR^2,\\
     &h(0, x)=h_0(x),\, {u}(0, x)= {u}_0(x),\,\, &x\in \RR^2.
\endaligned
\right.
\end{equation}
With respect to compressible Navier-Stokes equations under barotropic condition, which possess similar structure of system~\eqref{sw1},
there has a vast mathematical literature on well-poseness and ill-posedness results. In terms of the critical Besov spaces, it has been proved that compressible Navier-Stokes system is well-posed in critical Besov spaces $\dot B^{\frac{d}{p}}_{p, 1}(\RR^d)\times \dot B^{\frac{d}{p}-1}_{p, 1}(\RR^d)$ for $1\le p<2d$(see \cite{CMZ10, Dan00, Dan05}). And this system is ill-posed for $p\ge 2d$ (see \cite{CMZ15, IO21}). Following methods in \cite{CMZ08, CMZ10}, one can obtain that system \eqref{sw1} is  local well-posed for large initial data and global well-posed for small initial data in $\dot B^{\frac{2}{p}}_{p, 1}(\RR^2)\times \dot B^{\frac{2}{p}-1}_{p, 1}(\RR^2)$ for $1\le p<4$. Recently, Li, Hong and Zhu \cite{LHZ} proved system \eqref{sw1} is ill-posed for $p>4$. However, the question that whether system \eqref{sw1} is ill-posed or not for the endpoint case $p=4$ has not been answered.

In this paper, we aim to prove the ill-posedness of shallow water equations \eqref{sw1} for the endpoint case $p=4$. Motivated by \cite{CMZ15, IO21}, we construct initial data in the Schwartz class which are arbitrarily small in $\dot B^{\frac{1}{2}}_{4,1}(\RR^2)\times \dot B^{-\frac{1}{2}}_{4,1}(\RR^2)$, meanwhile the corresponding solutions are arbitrarily large in $\dot B^{\frac{1}{2}}_{4,1}(\RR^2)\times \dot B^{-\frac{1}{2}}_{4,1}(\RR^2)$ after an arbitrarily short time. This phenomenon shows that the solution map $(h_0, u_0)\mapsto (h[h_0], u[u_0])$ is discontinuous in $\dot B^{\frac{1}{2}}_{4,1}(\RR^2)\times \dot B^{-\frac{1}{2}}_{4,1}(\RR^2)$. Moreover, we observe that the special nonlinear mechanism and $L^2(\RR^2) \hookrightarrow\dot B^{-\frac{1}{2}}_{4, 2}(\RR^2)$ lead to the second iteration is continuous in $\dot B^{-\frac{1}{2}}_{4, 2}(\RR^2)$. Therefore,  we generalize the ill-posedness results and show that system \eqref{sw1} is ill-posed in $\dot B^{\frac{1}{2}}_{4,1}(\RR^2)\times \dot B^{-\frac{1}{2}}_{4,q}(\RR^2)$ for any $q\neq 2$. Our main result is as follows:

\begin{theorem}\label{thm}For $1\le q<2$, system \eqref{sw1} is ill-posed in critical spaces $\dot B^{\frac{1}{2}}_{4,1}(\RR^2)\times \dot B^{-\frac{1}{2}}_{4,q}(\RR^2)$. More precisely,  for any $\delta>0$ , there exists an initial data ${u}_0\in \dot B^{-\frac{1}{2}}_{4,q}\cap \mathcal{S}$ satisfying
\[h_0=0,\qquad\|u_0\|_{\dot B^{-\frac{1}{2}}_{4,q}}\le \delta, \]
such that the corresponding solution $(h, u)$ to system \eqref{sw1} satisfies
\[\|u(\cdot, t)\|_{\dot B^{-\frac{1}{2}}_{4,q}}> \frac{1}{\delta}, \quad \text{for some}\, \,0<t<\delta. \]
\end{theorem}
\begin{remark}
Generally speaking, researchers are focused on the well-posedness and ill-posedness of system \eqref{sw1} in critical Bseov spaces $\dot B^{\frac{2}{p}}_{p,1}(\RR^2)\times \dot B^{\frac{2}{p}-1}_{p,1}(\RR^2)$.
 Our result not only shows the ill-posedness of  system \eqref{sw1} in endpoint case $\dot B^{\frac{1}{2}}_{4,1}(\RR^2)\times \dot B^{-\frac{1}{2}}_{4,1}(\RR^2)$, but also generalizes the ill-posedness results in more critical Bseov spaces.
\end{remark}
\begin{theorem}\label{thm}For $q>2$, system \eqref{sw1} is ill-posed in critical spaces $\dot B^{\frac{1}{2}}_{4,1}(\RR^2)\times \dot B^{-\frac{1}{2}}_{4,q}(\RR^2)$. More precisely,  for any $\delta>0$ , there exists an initial data ${u}_0\in \dot B^{-\frac{1}{2}}_{4,q}\cap \mathcal{S}$ satisfying
\[h_0=0,\qquad\|u_0\|_{\dot B^{-\frac{1}{2}}_{4,q}}\le \delta, \]
such that the corresponding solution $(h, u)$ to system \eqref{sw1} satisfies
\[\|u(\cdot, t)\|_{\dot B^{-\frac{1}{2}}_{4,q}}> \frac{1}{\delta}, \quad \text{for some}\, \,0<t<\delta. \]
\end{theorem}

\begin{remark}By our method, the condition $q\neq 2$ is sharp, because the second iteration,  a bilinear operator,
\begin{equation}\label{B}
B(f,g):=-\int_0^t e^{(t-s)\Delta}(e^{s\Delta}f\cdot\nabla e^{s\Delta}g+\nabla \int_0^s \Div e^{\tau\Delta}f\dd \tau\cdot\nabla e^{s\Delta}g)\dd s
\end{equation}
satisfies that there exists an absolute constant $C$ such that for any $f,g\in \dot B^{-\frac{1}{2}}_{4,2}(\RR^2)$, and $t\ge 0$,
\[\|B(f,g)(t)\|_{\dot B^{-\frac{1}{2}}_{4,2}}\le C\|f\|_{\dot B^{-\frac{1}{2}}_{4,2}}\|g\|_{\dot B^{-\frac{1}{2}}_{4,2}}.\]
See Appendix for the proof. Therefore, extending our approach to the case $q=2$ would require new ideas and we will consider it later.
\end{remark}

\begin{remark}In \cite{IO21}, the second iteration of compressible Navier-Stokes equations is
\[\tilde{I}[f,g]=-\int_0^te^{c(t-s)\Delta}\{e^{As\Delta} f \nabla e^{Bs\Delta g}+\frac{B}{A}(e^{As\Delta}f-f)\nabla e^{Bs\Delta}g\}ds,\]
which is different from $B(f,g)$. In fact, their methods imply compressible Navier-Stokes equations is ill-posedness in $\dot B^{\frac{d}{p}}_{p,1}\times \dot B^{\frac{d}{p}}_{p,q}$ for $q<2$. We show system \eqref{sw1} is also ill-posed for $q>2$,  which implies that the mechanism between compressible of Navier-Stokes and system \eqref{sw1} is different.
\end{remark}
\section{Preliminaries}
\begin{lemma}[\cite{Ba}]\label{trans}Let $1\le p\le p_1\le\infty$ and $s\in (-2\min\{\frac{1}{p_1}, 1-\frac{1}{p}\}, 1+\frac{2}{p_1}].$ Let $v$ be a vector field such that $\nabla v\in L^1_T(\dot B^{\frac{2}{p_1}}_{p_1, 1}(\RR^2))$. There exists a constant $C$ depending on $p, s, p_1$ such that all solutions $f\in L^\infty_T(\dot B^s_{p,1}(\RR^2))$ of the transport equation
\[\partial_t f+v\cdot\nabla f=g,\,\,\,\,f(0,x)=f_0(x).\]
with initial data $f_0\in \dot B^s_{p,1}(\RR^2)$ and $g\in L^1_T(\dot B^s_{p,1}(\RR^2))$, we have, for $t\in[0, T]$,
\begin{align*}
\|f\|_{L^\infty_T(\dot B^s_{p,1})}\le e^{CV_{p_1}(t)}\Big(\|f_0\|_{\dot B^s_{p,1}}+\int_0^t  e^{-CV_{p_1}(\tau)}\|g(\tau)\|_{\dot B^s_{p,1}}\dd\tau\Big),
\end{align*}
where $V_{p_1}(t)=\int_0^t\|\nabla v\|_{\dot B^{\frac{2}{p_1}}_{p_1,1}(\RR^2)}\dd s$. Particularly, if $\nabla v\in L^1_T(\dot B^{0}_{\infty, 1}\cap \dot B^{\varepsilon}_{\infty, 1})$ for some $\varepsilon>0$, we have
\begin{align*}
\|f\|_{L^\infty_T(\dot B^0_{p,1})}\le e^{CV(t)}\Big(\|f_0\|_{\dot B^0_{p,1}}+\int_0^t  \|g(\tau)\|_{\dot B^0_{p,1}}\dd\tau\Big),
\end{align*}
where $V(t)=\int_0^t\|\nabla v\|_{\dot B^{0}_{\infty,1}}+\|\nabla v\|_{\dot B^{\varepsilon}_{\infty,1}}\dd s.$
\end{lemma}
\begin{lemma}[\cite{Dan05}]\label{heat}Let $s\in \RR$ and $1\le r_1, r_2, p, q\le \infty$ with $r_2\le r_1$. Consider the heat equation
\begin{align*}
\partial_t u-\Delta u=f,\qquad
  u(0,x)=u_0(x).
  \end{align*}
Assume that $u_0\in \dot B^s_{p,q}(\RR^2)$ and $f\in {L}^{r_2}_T(\dot B^{s-2+\frac{2}{r_2}}_{p,q}(\RR^2))$. Then the above equation has a unique solution $u\in {L}^{r_1}_T(\dot B^{s+\frac{2}{r_1}}_{p,q}(\RR^2))$
satisfying
\[\|u\|_{{L}^{r_1}_T(\dot B^{s+\frac{2}{r_1}}_{p,q})(\RR^2)}\le C\big(\|u_0\|_{\dot B^s_{p,q}(\RR^2)}+\|f\|_{{L}^{r_2}_T(\dot B^{s-2+\frac{2}{r_2}}_{p,q}(\RR^2))}\big).\]
\end{lemma}
\begin{lemma}[\cite{Ba}]\label{Paralinear}Let $s>0$ and $1\le r,p\le\infty$. Assume $F\in W^{[\sigma]+3}_{\text{loc}}(\RR)$ with $F(0)=0.$ Then for any $f\in L^\infty\cap \dot B^s_{p,1}$, we have
\[\|F(f)\|_{L^r_T(\dot B^s_{p,1})}\le C(1+\|f\|_{L^\infty_T(L^\infty)})^{[\sigma]+2}\|f\|_{L^r_T(\dot B^s_{p,1})}.\]
\end{lemma}
\begin{lemma}[\cite{IO21}]\label{bilinearest.}Define $\dot B^{\sigma}_{[2,\infty],1}:=\dot B^{\sigma}_{2,1}\cap \dot B^{\sigma}_{\infty,1}$. For $\sigma>0$ and $0<\e<1$, we have
\begin{align*}
&\|uv\|_{\B}\le C\min\{\|u\|_{\dot B^{\e}_{[2,\infty],1}}\|v\|_{\B}, \|u\|_{\B}\|v\|_{\dot B^{\e}_{[2,\infty],1}}\},\\
&\|uv\|_{\dot B^{\sigma}_{[2,\infty],1}}\le C(\|u\|_{\dot B^{\sigma}_{[2,\infty],1}}\|v\|_{\B}+ \|u\|_{\B}\|v\|_{\dot B^{\sigma}_{[2,\infty],1}}).
\end{align*}
\end{lemma}
\begin{lemma}[\cite{CM78}]\label{multi}
Let $m$ be a smooth function satisfying that
\[|\partial^\alpha_{\xi,\eta}m(\xi, \eta)|\le C_\alpha(|\xi|+|\eta|)^{-|\alpha|}\]
for all multi-index $\alpha$. Assume $p, p_1, p_2 \in (1, \infty)^3$ and $\frac{1}{p}=\frac{1}{p_1}+\frac{1}{p_2}$, then
\[\Big{\|}\mathscr{F}^{-1}\Big[\int_{\RR^d} m(\xi-\eta, \eta)\widehat{f}(\xi-\eta)\widehat{g}(\eta)\rd \eta\Big]\Big{\|}_{L^p}\le C\|f\|_{L^{p_1}}\|g\|_{L^{p_2}}.\]
\end{lemma}
\section{Ill-Posedness}
In this section, we construct special initial data $(h_0, u_0)$ and obtain ``norm inflation'' of the corresponding solution. In order to show the local existence and uniqueness of system~\eqref{sw1} for  given initial data, firstly we provide the following proposition which involves in some properties we needed later.
\begin{proposition}\label{exist}Fixed some $0<\e<\frac{1}{5}$, let $\delta>0$ such that $5\e+3\delta<1$ and $N$ be a large enough integer , if initial data $(h_0, u_0)\in \B\times\B$ is Schwarz function and satisfies
\[\|h_0\|_{\B}\le C2^{(-\frac{1}{2}+\frac{\delta}{2})N},\,\,\|h_0\|_{\Bo}\le C2^{(\frac{1}{2}+\frac{\delta}{2})N},\,\,\|u_0\|_{\B}\le C2^{(\frac{1}{2}+\frac{\delta}{2})N}\]
 there exist constants $C_0$ and $N_0$ such that for $N>N_0$ and $T=(\ln N)^{-1}2^{-2N}$, system~\eqref{sw1} has a unique local solution $(h, u)$ associated with initial data $(h_0, u_0)$ satisfying
\begin{align*}
&h\in C([0, T], \B)\cap L^{\infty}([0, T], \Bo),\\
&u\in C([0, T], \B)\cap L^{1}([0, T], \Bt),
\end{align*}
and the following estimates hold
\begin{equation}\label{bound}
\begin{aligned}
&\|h\|_{L^\infty_T(\B)}\le 2C^2_0e^{C_0}2^{(-\frac{1}{2}+\frac{\delta}{2})N},\\
&\|h\|_{L^\infty_T(\Bo)}\le 2C^2_02^{(\frac{1}{2}+\frac{\delta}{2})N},\\
&\|u\|_{L^\infty_T(\B)}+\|u\|_{L^1_T(\Bt)}\le C_02^{(\frac{1}{2}+\frac{\delta}{2})N}.\\
\end{aligned}
\end{equation}
\end{proposition}
\begin{proof}
\vskip 0.3mm
\textbf{First Step: Constructing Approximate Solutions}\\
Starting from $(h^0, {u}^0)=(0,0)$ and we define the approximate sequence $(\hn, \un)_{n\in\NN}$ of equations \eqref{sw1} by  solving the following linear system:
\begin{equation}\label{swn}
\left\{ \aligned
    & \partial_t \h+ \un\cdot\nabla\h=-(1+\hn)\Div \un, &t>0, x\in \RR^2,\\
     &\partial_t\uu-\Delta\uu=-\un\cdot\nabla\un-\nabla\hn+\nabla(\ln(1+\hn))\cdot\nabla\un, &t>0, x\in \RR^2,\\
     &\h(0, x)=h_0,\uu(0, x)={u}_0(x), &x\in \RR^2.
\endaligned
\right.
\end{equation}
\noindent\textbf{Second Step: Uniform Bounds}

It is easy to check that  $(h^1, u^1)=(h_0, e^{t\Delta}u_0)$ and $h_1$ satisfies  estimates in \eqref{bound}. For $u_1$, there exists a constant ~$C_0>1$ such that

\begin{align*}
&\|u^1\|_{L^\infty_T(\B)}+\|u^1\|_{L^1_T(\Bt)}\le C\|u_0\|_{\B}\le {C_0}2^{(\frac{1}{2}+\frac{\delta}{2})N}.
\end{align*}
Assume estimates \eqref{bound} hold for $(\hn, \un)$, next we check it for $(\h, \uu)$. With aid of Lemma \ref{trans} and Lemma \ref{bilinearest.}, we obtain that
\begin{align*}
\|\h\|_{L^\infty_T(\B)}\le e^{C\|\un\|_{L^1_T(\dot B^1_{\infty, 1}\cap \dot B^{1+\e}_{\infty, 1})}}\int_0^T(\|\hn\|_{\B}\|\un\|_{ \dot B^{1+\e}_{[2,\infty],1}}+\|\un\|_{\Bo})\dd t.
\end{align*}
Noting the fact that
\begin{align*}
\|\un\|_{L^1_T(\dot B^1_{\infty, 1}\cap \dot B^{1+\e}_{\infty, 1})}\le &T^{\frac{1}{2}}\|\un\|_{L^2_T(\dot B^1_{\infty, 1})}+T^{\frac{1-\e}{2}}\|\un\|_{L^{\frac{2}{1+\e}}_T(\dot B^{1+\e}_{\infty, 1})}\\
\le&  C_02^{(\frac{1}{2}+\frac{\delta}{2})N}\cdot (\ln N)^{-\frac{1}{2}} 2^{-N}+C_02^{(\frac{1}{2}+\frac{\delta}{2})N}\cdot (\ln N)^{-\frac{1-\e}{2}} 2^{(-1+\e)N}\\
\le& 2C_0(\ln N)^{-\frac{1-\e}{2}}2^{(-\frac{1}{2}+\frac{\delta}{2}+\e)N},
\end{align*}
and
\begin{align*}
\int_0^T\|\hn\|_{\B}\|\un\|_{\dot B^{1+\e}_{[2,\infty],1}}\dd t
\le& T^{\frac{1-\e}{2}}\|\hn\|_{L^\infty_T(\B)}\|\un\|_{L^{\frac{2}{1+\e}}_T\dot B^{1+\e}_{[2,\infty],1}}\\
\le & 2 C^3_0e^{C_0}(\ln N)^{-\frac{1-\e}{2}}2^{({\delta+\e-1})N},
\end{align*}
it yields that for $3\delta+5\e<1$, there exists a $N_0$ such that for $N>N_0$,
\begin{align*}
\|\h\|_{L^\infty_T(\B)}\le &e^{ C_02^{(-\frac{1}{2}+\frac{\delta}{2}+\e)N}}(2 C^3_0e^{C_0}(\ln N)^{-\frac{1-\e}{2}}2^{({\delta+\e-1})N}+C_0(\ln N)^{-\frac{1}{2}}2^{(-\frac{1}{2}+\frac{\delta}{2})N})\\
\le&2C^2_0e^{C_0}2^{(-\frac{1}{2}+\frac{\delta}{2})N}.
\end{align*}
Similarly, for $3\delta+5\e<1$,  $N>N_0$ and $C_0>C$, we have
\begin{equation}\label{hB1}
\begin{aligned}
&\|\h\|_{L^\infty_T(\Bo)}\\
\le& e^{C\|\un\|_{L^1_T(\dot B^1_{\infty, 1})}}\int_0^T(\|\hn\|_{\Bo}\|\un\|_{\Bo}+\|\un\|_{\Bt}+\|\hn\|_{\B}\|\un\|_{\Bt})\dd t\\
\le& e^{C_02^{(-\frac{1}{2}+\frac{\delta}{2})N}}(2 C^3_0(\ln N)^{-\frac{1}{2}}2^{(\delta-1)N}+C_02^{(\frac{1}{2}+\frac{\delta}{2})N}+2C_0^3e^{C_0}2^{\delta N})
\le C^2_02^{(\frac{1}{2}+\frac{\delta}{2})N}.
\end{aligned}
\end{equation}
Using {Lemma} \ref{heat}, we obtain that
\begin{align*}
&\|\uu\|_{L^\infty_T(\B)}+\|\uu\|_{L^1_T(\Bt)}\\
\le& \|u_0\|_{\B}+\|\un\cdot\nabla\un\|_{L^1_T(\B)}+\|\nabla\hn\|_{L^1_T(\B)}+\|\nabla(\ln(1+\hn))\cdot\nabla\un\|_{L^1_T(\B)},
\end{align*}
where for $N>N_0$, it holds that
\begin{align*}
\|\un\cdot\nabla\un\|_{L^1_T(\B)}\le &T^{\frac{1-\e}{2}}\|\un\|_{L^{\frac{2}{\e}}_T(\dot B^\e_{[2,\infty], 1})}\|\un\|_{L^2_T(\Bo)}\le C^2_0 2^{(\delta+\e)N}\le 2C^2_0(\ln N)^{-\frac{1-\e}{2}}2^{(\e+\delta)N},
\end{align*}
and
\begin{align*}
\|\nabla\hn\|_{L^1_T(\B)}\le T\|\hn\|_{L^\infty_T(\Bo)}\le 2C^2_0(\ln N)^{-1}2^{(\frac{\delta}{2}-\frac{3}{2})N}.
\end{align*}
Owning to Lemma \ref{Paralinear}, one yields that
\begin{align*}
\|\nabla(\ln(1+\hn))\cdot\nabla\un\|_{L^1_T(\B)}\le& T^{\frac{1-\e}{2}}\|\ln(1+\hn)\|_{L^\infty_T(\Bo)}\|\un\|_{L^{\frac{2}{1+\e}}_T(\dot B^{1+\e}_{[2,\infty], 1})}\\
\le& 2CC^3_0(1+\|\hn\|_{L^\infty_T(L^\infty)})^2(\ln N)^{-\frac{1-\e}{2}}2^{(\delta+\e)N}\\
\le &8CC^3_0(\ln N)^{-\frac{1-\e}{2}}2^{(\delta+\e)N}.
\end{align*}
Therefore, for some fixed $C_0>1$, we can find some $N_0$ such that for any $N\ge N_0$, $(\hn, \un)_{n\in\NN}$ satisfies uniformly estimates \eqref{bound}.
\vskip 0.3mm
\noindent \textbf{Third Step: Time Derivatives}
\vskip 0.3mm
Furthermore, the sequence $(h^n, u^n)_{n\ge 0}$ is uniformly bounded in
\[C^{\frac{1}{2}}([0, T]; B^{0}_{[2, \infty],1})\times C^{\frac{2-\varepsilon}{2}}([0, T]; B^{-\varepsilon}_{[2, \infty],1}).\]
Indeed, $(\hn, \un)_{n\ge 0}$ possesses uniformly bounds  \eqref{bound} and
\[\partial_th^{n+1}=-u^n\cdot\nabla h^{n+1}-(1+h^n)\Div u^n,\]
the right-hand side is uniformly bounded in $L^2_T(B^0_{[2, \infty], 1})$.

As regards to $(u^n)_{n\ge 0}$, this follows from the fact that
\[\partial_t\uu=\Delta\uu-\un\cdot\nabla\un-\nabla\hn+\nabla(\ln(1+\hn))\cdot\nabla\un.\]
by using the fact that $(\un)_{n\ge 0}$ and $(\hn)_{n\ge 0}$ are uniformly bounded in $L^\infty_T (\B)\cap L^1_T(\Bt)$ and $ L^\infty_T (\B\cap\Bo)$, we easily deduce that the four terms on the right-hande side are in ${L^{\frac{2}{2-\varepsilon}}_T( B^{-\varepsilon}_{[2, \infty],1})}.$
\vskip 0.3mm
\noindent \textbf{Fourth Step: Convergence and Uniqueness}
\vskip 0.3mm
 Let $\{\phi_m\}_{m\in\NN}$ be a sequence of smooth functions with values in $[0, 1]$, supported in the ball $B(0, m+1)$ and equal to $1$ on $B(0, m)$. Taking advantage the uniformly estimates on $(\hn, \un)_{n\ge 0}$, by Aubin-Lions lemma and the Cantor diagonal process, we obtain  that there exists  a subsequence of $(\hn, \un)_{n\ge 0}$ (still denoted by $(\hn, \un)_{n\ge 0}$) such that, for all $m\in\NN$,
\begin{equation}\label{cov2}
 \begin{aligned}
(\phi_m \hn, \phi_m \un)\rightarrow (\phi_m h, \phi_m u) \,\,\text{in}\,\, C([0, T]; B^{0}_{[2, \infty],1}\times B^{-\varepsilon}_{[2, \infty],1}).
 \end{aligned}
 \end{equation}
 Therefore, $(\hn, \un)$ tends to $(h, u)$ in $\mathcal{D}'([0, T]\times\RR^2)$. Following the argument in \cite{Dan00},   it is routine to verify that $(h ,  {u})$ satisfies  system \eqref{sw1} and the solution is continuous in terms of time in $\B(\RR^2)\times\B(\RR^2)$. Readers can refer to \cite{CMZ08} to prove the uniqueness.
\end{proof}

\begin{proposition}\label{minius}
Let $\e$ and $\delta$ be defined in Proposition \ref{exist}. The solution $(h, u)$ obtained in Proposition \ref{exist} satisfies the following estimates:
\begin{equation}
\|h+\int_0^t\Div U_0\dd s\|_{L^\infty_T(\dot B^1_{[2,\infty],1})}\le C2^{{(\e+\delta-1)N}}
\end{equation}
\begin{equation}\label{u-u0}
\|u-U_0\|_{L^\infty_T (\dot B^0_{[2,\infty],1})}+\|u-U_0\|_{L^1_T (\dot B^2_{[2,\infty],1})}
\le C2^{(\varepsilon+\delta)N}.
\end{equation}
\end{proposition}
\begin{proof}Taking advantage of Lemma \ref{heat} and then we obtain
\begin{align*}
&\|u-U_0\|_{L^\infty_T (\dot B^0_{[2,\infty],1})}+\|u-U_0\|_{L^1_T (\dot B^2_{[2,\infty],1})}\\
\le& C(\|u\cdot\nabla u\|_{L^1_T (\dot B^0_{[2,\infty],1})}+\|\nabla h\|_{L^1_T (\dot B^0_{[2,\infty],1})}+\|\nabla(\ln(1+h))\cdot\nabla u\|_{L^1_T(\dot B^0_{[2,\infty],1})})\\
\le& CT^{\frac{1-\varepsilon}{2}}\|u\|_{L^\infty_T(\dot B^0_{[2,\infty],1})}\| u\|_{L^{\frac{2}{1+\varepsilon}}_T (\dot B^{1+\varepsilon}_{[2,\infty],1})}+CT\|h\|_{L^\infty_T (\dot B^1_{[2,\infty],1})}\\
&+CT^{\frac{1-\varepsilon}{2}}(1+\|h\|_{L^\infty})^2\|h\|_{L^\infty_T(\dot B^1_{[2,\infty],1})}\| u\|_{L^{\frac{2}{1+\varepsilon}}_T (\dot B^{1+\varepsilon}_{[2,\infty],1})}\\
\le& C2^{(\varepsilon+\delta)N}.
\end{align*}
Based on the above inequality, it holds that
\begin{align*}
\|h+\int_0^t\Div U_0\dd s\|_{L^\infty_T(\dot B^0_{[2,\infty],1})}
\le&\big\|\int_0^t(\Div(u-U_0)+\Div(hu))\dd s\big\|_{L^\infty_T(\dot B^0_{[2,\infty],1})}\\
\le&\|u-U_0\|_{L^1_T(\dot B^1_{[2,\infty],1})}+\|hu\|_{L^1_T(\dot B^1_{[2,\infty],1})}
\le C2^{{(\e+\delta-1)N}}.
\end{align*}
\end{proof}
\subsection{Proof of Theorem \ref{thm} for \bm{$1\le q< 2$}.}
Let $(\varphi_j)_{j\in\ZZ}$ be the Littlewood-Paley convolution functions. We introduce
\[\Phi_{j,N}=\varphi_0(2^j(x-2^{|j|+2N}e_1)),\]
and initial data $(h_0, u_0)$ is defined by:
\[h_0=0,\quad u_0=\Big(\frac{ 2^{\frac{N}{2}}}{N^{\frac{1}{4}}\ln N}\sum_{-\delta N\le j\le 0}2^{\frac{1}{2}j}\Phi_{j,N}\sin (2^N x_1), \frac{ 2^{\frac{N}{2}}}{N^{\frac{1}{4}}\ln N}\sum_{-\delta N\le j\le 0}2^{\frac{1}{2}j}\Phi_{j,N}\cos (2^N x_1)\Big),\]
where $\delta$ is consistent with that in Proposition \ref{exist}. It is easy to check that supp $\widehat{u}_0(\xi)\subset\{\xi\in\RR^2|2^{N-1}\le|\xi|\le 2^{N+1}\}.$ Hence
\begin{align*}
\|u_0\|_{\dot B^{-\frac{1}{2}}_{4,q}}\le& \frac{C}{N^{\frac{1}{4}}\ln N}\Big\|\sum_{-\delta N\le j\le 0}2^{\frac{1}{2}j}\Phi_{j,N}\Big\|_{L^4}.
\end{align*}
An easy computation yields that
\begin{align*}
&\Big\|\sum_{-\delta N\le j\le 0}2^{\frac{1}{2}j}\Phi_{j,N}\Big\|^4_{L^4}
\le \sum_{-\delta N\le j\le 0}\int_{\RR^2}2^{2j}|\Phi_{j,N}|^4\dd x+\sum_{(j_1,\cdots, j_{4})\in \Lambda}\int_{\RR^2}\Big|2^{\frac{1}{2}{(j_1+\dots+j_{4})}} \Phi_{j_1,N}\cdots\Phi_{j_4,N}\Big|\dd x,
\end{align*}
where the set $\Lambda$ is defined by
\[\Lambda=\big\{(j_1,\dots, j_{4})\in [-N, 0]^{4}\cap\NN^{4}\vert \exists 1\le k, \ell\le 4 \,\,\text{s.t.}-N\le j_{k}\neq j_{\ell}\le 0 \big\}.\]
From the definition of $\Phi_{j,N}$, we obtain that
\begin{equation}\label{Phij}
\begin{aligned}
\sum_{-\delta N\le j\le 0}\int_{\RR^2}2^{2j}|\Phi_{j,N}|^4\dd x=\sum_{-\delta N\le j\le 0}\int_{\RR^2}|\varphi_0|^4\dd x\le CN.
\end{aligned}
\end{equation}
Due to $\Phi_{j, N}\in\mathcal{S}$, for any $k>0$, there exists a constant $C_k$ such that
$$|\Phi_{j, N}|\le C_k(1+2^j|x-2^{|j|+2N}e_1|)^{-k}.$$
Assume $j_1\neq j_2$,  from the above inequality, we have that
\begin{align*}
&\int_{\RR^d}\Big|\Phi_{j_1,N}\cdots\Phi_{j_4,N}\Big|\dd x\\
\le&C_k\int_{\RR^d}\frac{1}{(1+2^{j_1}|x-2^{|j_1|+2N}e_1|)^k}
\frac{1}{(1+2^{j_2}|x-2^{|j_2|+2N}e_1|)^k}\dd x\\
\le&C_k\Big(\int_{|x-2^{|j_1|+2N}e_1|\le\frac{1}{2}2^{2N}}\frac{1}{(1+2^{j_1}|x-2^{|j_1|+2N}e_1|)^k}
\frac{1}{(1+2^{j_2}|x-2^{|j_2|+2N}e_1|)^k}\dd x\\
&+\int_{|x-2^{|j_1|+2N}e_1|>\frac{1}{2}2^{2N}}\frac{1}{(1+2^{j_1}|x-2^{|j_1|+2N}e_1|)^k}
\frac{1}{(1+2^{j_2}|x-2^{|j_2|+2N}e_1|)^k}\dd x\Big).
\end{align*}
Noting the fact that $||j_1|-|j_2||\ge 1$, for $|x-2^{|j_1|+2N}e_1|\le\frac{1}{2}2^{2N}$, by triangle inequality we obtain that
\begin{align*}
2^{j_2}|x-2^{|j_2|+2N}e_1|
\ge&2^{j_2}|2^{|j_1|+2N}e_1-2^{|j_2|+2N}e_1|-\frac{1}{2}2^{j_2}2^{2N}\\
\ge&2^{j_2}2^{2N}-\frac{1}{2}2^{j_2}2^{2N}\ge\frac{1}{2}2^{j_2}2^{2N}.
\end{align*}
Therefore, taking $k>2$, one yields that
\begin{align*}
&\int_{|x-2^{|j_1|+2N}e_1|\le\frac{1}{2}2^{2N}}\frac{1}{(1+2^{j_1}|x-2^{|j_1|+2N}e_1|)^k}
\frac{1}{(1+2^{j_2}|x-2^{|j_2|+2N}e_1|)^k}\dd x\\
\le &C2^{-(j_2+2N-1)k}\int_{|x-2^{|j_1|+2N}e_1|\le\frac{1}{2}2^{2N}}\frac{1}{(1+2^{j_1}|x-2^{|j_1|+2N}e_1|)^k}\dd x\\
\le&C2^{-(j_2+2N-1)k}2^{-2j_1}\le C.
\end{align*}
Similarly, we have
\begin{align*}
&\int_{|x-2^{|j_1|+2N}e_1|>\frac{1}{2}2^{2N}}\frac{1}{(1+2^{j_1}|x-2^{|j_1|+2N}e_1|)^k}
\frac{1}{(1+2^{j_2}|x-2^{|j_2|+2N}e_1|)^k}\dd x\\
\le& C2^{-(j_1+2N-1)k}2^{-2j_2}\le C.
\end{align*}
Hence, by \eqref{Phij} and the above two inequalities, one gets
\begin{align*}
\|u_0\|_{\dot B^{-\frac{1}{2}}_{4,q}}\le \frac{C}{N^{\frac{1}{4}}\ln N}(N+\sum_{(j_1,\cdots, j_{4})\in \Lambda}2^{\frac{1}{2}{(j_1+\dots+j_{4})}} )^{\frac{1}{4}}\le \frac{C}{\ln N}.
\end{align*}
Now we decompose the solution $u$ into three parts:
\begin{equation}\label{decomposition}
u=U_0+U_1+U_2,
\end{equation}
where $U_0=e^{t\Delta}u_0$, and\\
\begin{align*}
U_1:=-\int_0^t e^{(t-s)\Delta}(U_0\cdot\nabla U_0+\nabla \int_0^s \Div U_0(\tau)\dd \tau\cdot\nabla U_0(s)-\nabla h_0\cdot\nabla e^{s\Delta}u_0)\dd s.
\end{align*}
Next, we estimate $\|U_i(t_0)\|_{\dot B^{-\frac{1}{2}}_{4, 1}}(i=0,1,2)$ respectively, where $t_0=(\ln N)^{-1}2^{-2N}$.

\noindent\emph{{Estimates on $\|U_0(t_0)\|_{\dot B^{-\frac{1}{2}}_{4, 1}}$}}.\,\, From the above estimates on $\|u_0\|_{\dot B^{-\frac{1}{2}}_{4,q}}$, it yields that
\begin{equation}\label{U_0}
\|U_0(t_0)\|_{\dot B^{-\frac{1}{2}}_{4, q}}\le C\|u_0\|_{\dot B^{-\frac{1}{2}}_{4, q}}\le \frac{C}{\ln N}.
\end{equation}
We denote the $k-th$ component of $U_1$ by $U^{(k)}_1$, then
\begin{align*}
U^{(2)}_1=-\int_0^t e^{(t-s)\Delta}(U_0(s)\cdot\nabla U^{(2)}_0(s)+\nabla \int_0^s \Div U_0(\tau)\dd \tau\cdot\nabla U^{(2)}_0(s))\dd s.
\end{align*}
\emph{{Estimates on $\|U_1(t_0)\|_{\dot B^{-\frac{1}{2}}_{4, q}}$}}.\,\, By the definition of $U_1$ and $h_0=0$, we obtain that
\begin{equation}\label{U_1}
\begin{aligned}
&\|U_1(t_0)\|_{\dot B^{-\frac{1}{2}}_{4, q}}\ge \|U^{(2)}_1(t_0)\|_{\dot B^{-\frac{1}{2}}_{4, q}}\ge \Big(\sum_{-\delta N\le j\le 0}2^{-\frac{1}{2}jq}\|\dot\Delta_jU^{(2)}_1(t_0)\|^q_{L^4}\Big)^{\frac{1}{q}}\\
\ge&\Big(\sum_{-\delta N\le j\le 0}2^{-\frac{1}{2}jq}\Big{\|}\dot\Delta_j\int_0^t e^{(t-s)\Delta}U_0\cdot\nabla U^{(2)}_0\dd s\Big{\|}^q_{L^4}\Big)^{\frac{1}{q}}\\
&-\Big(\sum_{-\delta N\le j\le 0}2^{-\frac{1}{2}jq}\Big{\|}\dot\Delta_j\int_0^t e^{(t-s)\Delta}\nabla \int_0^s \Div U_0(\tau)\dd \tau\cdot\nabla U^{(2)}_0(s)\dd s\Big{\|}^q_{L^4}\Big)^{\frac{1}{q}}\\
\ge& \Big(\sum_{-\delta N\le j\le 0}2^{-\frac{1}{2}jq}\big\|\dot\Delta_j\int_0^{t_0} e^{(t_0-s)\Delta}(U^1_0\partial_{x_1} U^2_0)\dd s\big\|^q_{L^4}\Big)^{\frac{1}{q}}\\
&-\Big(\sum_{-\delta N\le j\le 0}2^{-\frac{1}{2}jq}\big\|\dot\Delta_j\int_0^{t_0} e^{(t_0-s)\Delta}(U^2_0\partial_{x_2} U^2_0)\dd s\big\|^q_{L^4}\Big)^{\frac{1}{q}}\\
&-\Big(\sum_{-\delta N\le j\le 0}2^{-\frac{1}{2}jq}\Big{\|}\dot\Delta_j\int_0^t e^{(t-s)\Delta}\big(\nabla \int_0^s \Div U_0(\tau)\dd \tau\cdot\nabla U^{(2)}_0(s)\big)\dd s\Big{\|}^q_{L^4}\Big)^{\frac{1}{q}}.
\end{aligned}
\end{equation}
At the beginning, we give the upper bound of the second term on the right-hand side of the above inequality. Using Bernstein's inequality, we have
\begin{align*}
&\Big(\sum_{-\delta N\le j\le 0}2^{-\frac{1}{2}jq}\big\|\dot\Delta_j\int_0^{t_0} e^{(t_0-s)\Delta}(U^2_0\partial_{x_2} U^2_0)\dd s\big\|^q_{L^4}\Big)^{\frac{1}{q}}
\le\Big(\sum_{-\delta N\le j\le 0}2^{jq}\Big)^{\frac{1}{q}}t_0\|U_0\|^2_{L^\infty_{t_0}L^4}\le \frac{C}{(\ln N)^2}.
\end{align*}
By Fourier transform, it yields that
\begin{align*}
&\Big(\sum_{-\delta N\le j\le 0}2^{-\frac{1}{2}jq}\big\|\dot\Delta_j\int_0^{t_0} e^{(t_0-s)\Delta}(U^1_0\partial_{x_1} U^2_0)\dd s\big\|^q_{L^4}\Big)^{\frac{1}{q}}\\
=&\Big(\sum_{-\delta N\le j\le 0}2^{-\frac{1}{2}jq}\big\|\mathscr{F}^{-1}\big(\phi_j(\xi)\int_{\RR^2}
\frac{e^{-t_0|\xi|^2}-e^{-t_0(|\xi-\eta|^2+|\eta|^2)}}{|\xi-\eta|^2+|\eta|^2-|\xi|^2}\widehat{u^1_0}(\xi-\eta)i\eta_1 \widehat{u^2_0}(\eta)\dd \eta\big)\big\|^q_{L^4}\Big)^{\frac{1}{q}}.
\end{align*}
Taking advantage of Taylor's series, we have
\begin{align*}
&\frac{e^{-t_0|\xi|^2}-e^{-t_0(|\xi-\eta|^2+|\eta|^2)}}{|\xi-\eta|^2+|\eta|^2-|\xi|^2}
=t_0e^{-t_0|\xi|^2}\sum_{k=1}^\infty(-1)^{k+1}\frac{\big(t_0(|\xi-\eta|^2+|\eta|^2-|\xi|^2)\big)^{k-1}}{k!}\\
=&t_0e^{-t_0|\xi|^2}+t_0e^{-t_0|\xi|^2}\sum_{k=2}^\infty(-1)^{k+1}\frac{\big(t_0(|\xi-\eta|^2+|\eta|^2-|\xi|^2)\big)^{k-1}}{k!}
:=t_0e^{-t_0|\xi|^2}+G(\xi-\eta,\eta).
\end{align*}
Hence, it follows from the first term on the right-hand side of inequality \eqref{U_1} that
\begin{align*}
&\Big(\sum_{-\delta N\le j\le 0}2^{-\frac{1}{2}jq}\big\|\dot\Delta_j\int_0^{t_0} e^{(t_0-s)\Delta}(U^1_0\partial_{x_1} U^2_0)\dd s\big\|^q_{L^4}\Big)^{\frac{1}{q}}\\
\ge&\Big(\sum_{-\delta N\le j\le 0}2^{-\frac{1}{2}jq}\big\|\mathscr{F}^{-1}\big(\phi_j(\xi)
t_0\int_{\RR^2}\widehat{u^1_0}(\xi-\eta)i\eta_1 \widehat{u^2_0}(\eta)\dd \eta\big)\big\|^q_{L^4}\Big)^{\frac{1}{q}}\\
&-\Big(\sum_{-\delta N\le j\le 0}2^{-\frac{1}{2}jq}\big\|\mathscr{F}^{-1}\big(\phi_j(\xi)
(t_0e^{-t_0|\xi|^2}-t_0)\int_{\RR^2}\widehat{u^1_0}(\xi-\eta)i\eta_1 \widehat{u^2_0}(\eta)\dd \eta\big)\big\|^q_{L^4}\Big)^{\frac{1}{q}}\\
&-\Big(\sum_{-\delta N\le j\le 0}2^{-\frac{1}{2}jq}\big\|\mathscr{F}^{-1}\big(\phi_j(\xi)
\int_{\RR^2}G(\xi-\eta,\eta)\widehat{u^1_0}(\xi-\eta)i\eta_1 \widehat{u^2_0}(\eta)\dd \eta\big)\big\|^q_{L^4}\Big)^{\frac{1}{q}}\\
:=&\big(\sum_{-\delta N\le j\le 0}2^{-\frac{1}{2}jq}I^q_j\big)^{\frac{1}{q}}-\big(\sum_{-\delta N\le j\le 0}2^{-\frac{1}{2}jq}II^q_j\big)^{\frac{1}{q}}-\big(\sum_{-\delta N\le j\le 0}2^{-\frac{1}{2}jq}III^q_j\big)^{\frac{1}{q}}.
\end{align*}
For the term $I_j$, from the definition of initial data $u_0$, we  get that
\begin{align*}
I_j
\ge&\frac{t_02^{2N}}{(\ln N)^2N^{\frac{1}{2}}}2^j\|\dot\Delta_j(\Phi^2_{j, N}\sin^2(2^Nx_1))\|_{L^4}\\
&-\frac{t_02^{2N}}{(\ln N)^2N^{\frac{1}{2}}}\sum_{-\delta N\le k\neq j\le0}2^{k}\|\dot\Delta_j(\Phi^2_{k, N}\sin^2(2^Nx_1))\|_{L^4}\\
&-\frac{t_02^{2N}}{(\ln N)^2N^{\frac{1}{2}}}\sum_{-\delta N\le k\neq m\le0}2^{\frac{k}{2}+\frac{m}{2}}\|\dot\Delta_j(\Phi_{k, N}\Phi_{m, N}\sin^2(2^Nx_1))\|_{L^4}\\
&-\frac{t_02^{N}}{(\ln N)^2N^{\frac{1}{2}}}\sum_{-\delta N\le k, m\le0}2^{\frac{k}{2}+\frac{m}{2}}\|\dot\Delta_j(\Phi_{k, N}\partial_{x_1}(\Phi_{m, N})\sin(2^Nx_1)\cos(2^N x_1))\|_{L^4}\\
:=&I_{j1}-I_{j2}-I_{j3}-I_{j4}.
\end{align*}
Now we estimate $I_{ji}(i=1,2,3,4)$ respectively. We define the set $E_j$ by
\[E_j=\{x\in\RR^2| |x-2^{|j|+2N}e_1|\le 2^{-j}\}.\]
For $I_{j1}$, by $\sin^2x=\frac{1-\cos 2x}{2}$ and  triangle inequality, we have
\begin{align*}
I_{j1}&\ge \frac{t_02^{2N}}{2(\ln N)^2N^{\frac{1}{2}}}2^j\big{\|}\int_{\RR^2}2^{2j}\varphi_0(2^j(x-y))\varphi^2_0(2^j(y-2^{|j|+2N}e_1))\dd y\big{\|}_{L^4(E_j)}\\
&-\frac{t_02^{2N}}{2(\ln N)^2N^{\frac{1}{2}}}2^j\big{\|}\int_{\RR^2}2^{2j}\varphi_0(2^j(x-y))\varphi^2_0(2^j(y-2^{|j|+2N}e_1))\cos(2^{N+1}y_1)\dd y\big{\|}_{L^4(E_j)}.
\end{align*}
Taking advantage of change of variables, the first term on the right-hand side of the above inequality can be bounded as follows.
\begin{align*}
 &\frac{t_02^{2N}}{2(\ln N)^2N^{\frac{1}{2}}}2^j\big{\|}\int_{\RR^2}2^{2j}\varphi_0(2^j(x-y))\varphi^2_0(2^j(y-2^{|j|+2N}e_1))\dd y\big{\|}_{L^4(E_j)}\\
 =&\frac{t_02^{2N}2^{\frac{j}{2}}}{2(\ln N)^2N^{\frac{1}{2}}}\big{\|}\int_{\RR^2}\varphi_0(x-y)\varphi^2_0(y)\dd y\big{\|}_{L^4(|x|\le 1)}=\frac{C 2^{\frac{j}{2}}}{(\ln N)^3N^{\frac{1}{2}}}.
\end{align*}
By integration by parts, one yields that
\begin{align*}
&\frac{t_02^{2N}}{2(\ln N)^2N^{\frac{1}{2}}}2^j\big{\|}\int_{\RR^2}2^{2j}\varphi_0(2^j(x-y))\varphi^2_0(2^j(y-2^{|j|+2N}e_1))\cos(2^{N+1}y_1)\dd y\big{\|}_{L^4(E_j)}\\
=&\frac{t_02^{2N}2^{3j}}{2(\ln N)^2N^{\frac{1}{2}}}\big{\|}\int_{\RR^2}\partial_{y_1}(\varphi_0(2^j(x-y))\varphi^2_0(2^j(y-2^{|j|+2N}e_1)))\frac{\sin(2^{N+1}y_1)}{2^{N+1}}\dd y\big{\|}_{L^4(E_j)}\\
\le&\frac{C2^{-N}2^{\frac{3j}{2}}}{(\ln N)^3N^{\frac{1}{2}}}.
\end{align*}
Owning to $-\delta N\le j\le0$, from the above two estimates,  we infer that
\[I_{j1}\ge \frac{C 2^{\frac{j}{2}}}{(\ln N)^3N^{\frac{1}{2}}}.\]
Noting the fact that $|\varphi_0(x)|\le\frac{C_\beta}{(1+|x|)^\beta}$  for $\beta\in \NN$, $I_{j2}$ can be bounded by
\begin{align*}
I_{j2}\le& \frac{t_02^{2N}}{(\ln N)^2N^{\frac{1}{2}}}\sum_{-\delta N\le k\neq j\le0}2^{k}\|2^{2j}\int_{\RR^d}|\varphi_0(2^j(x-y))|\varphi^2_0(2^k(y-2^{|k|+2N}e_1))\dd y\|_{L^4(E_j)}\\
\le& \frac{Ct_02^{2N}}{(\ln N)^2N^{\frac{1}{2}}}\sum_{-\delta N\le k\neq j\le0}2^{k+2j}\Big\|\int_{\RR^d}\frac{1}{(1+2^j|x-y|)^\beta}\frac{1}{(1+2^k|y-2^{|k|+2N}e_1|)^{2\beta}}\dd y\Big\|_{L^4(E_j)}.
\end{align*}
Dividing the integral region in terms of $y$ into the following three parts to estimate:
\begin{align*}
&A_1:=\{y| |y-2^{|j|+2N}e_1|\le2^{2N-1}\},\\
&A_2:=\{y| |y-2^{|j|+2N}e_1|\ge2^{2N-1}, |y-2^{|k|+2N}e_1|\ge2^{2N-2}\},\\
&A_3:=\{y| |y-2^{|j|+2N}e_1|\ge2^{2N-1}, |y-2^{|k|+2N}e_1|\le 2^{2N-2}\},
\end{align*}
we conclude that,  for $x\in E_j$ and $y\in A_1$,
\begin{align*}
|y-2^{|k|+2N}e_1|&=|y-2^{|j|+2N}e_1+2^{|j|+2N}e_1-2^{|k|+2N}e_1| \\
&\ge |2^{|j|+2N}e_1-2^{|k|+2N}e_1|-|y-2^{|j|+2N}e_1|\ge 2^{2N-1}.
\end{align*}
For $x\in E_j$, $y\in A_3$, it is easy to check that
\begin{align*}
|x-y|=&|x-2^{|j|+2N}e_1+2^{|j|+2N}e_1-2^{|k|+2N}e_1+2^{|k|+2N}e_1-y|\\
\ge &|2^{|j|+2N}e_1-2^{|k|+2N}e_1|-|y-2^{|k|+2N}e_1|-2^{-j}\ge C2^{2N}.
\end{align*}
Therefore, for $-\delta N\le j\le0$, we obtain that
\begin{align*}
I_{j2}\le& \frac{Ct_02^{2N}}{(\ln N)^2N^{\frac{1}{2}}}\sum_{-\delta N\le k\neq j\le0}2^{k+2j}\Big(\Big\|\frac{2^{-2j}}{(2^k2^{2N})^{2\beta}}\Big\|_{L^{4}(A_j)}
+\Big\|\frac{2^{-2k}}{(2^j2^{2N})^\beta}\Big\|_{L^{4}(A_j)}\Big)
\le \frac{C2^{-2N}2^{-\frac{1}{2}j}}{(\ln N)^3N^{\frac{1}{2}}}.
\end{align*}
Following the similar arguments, we deduce that
\begin{align*}
I_{j3}\le& \frac{t_02^{2N}}{(\ln N)^2N^{\frac{1}{2}}}\sum_{-\delta N\le k\neq m\le0}2^{\frac{1}{2}(k+m)}  \|\varphi_0(2^k(x-2^{|k|+2N}e_1))\varphi_0(2^m(x-2^{|m|+2N}e_1))\|_{L^{4}(\RR^2)}\\
\le& \frac{Ct_02^{2N}}{(\ln N)^2N^{\frac{1}{2}}}\sum_{-\delta N\le k\neq m\le 0}2^{\frac{1}{2}(k+m)} \Big\|\frac{1}{(1+2^k|x-2^{|k|+2N}e_1|)}\frac{1}{(1+2^m|x-2^{|m|+2N}e_1|)}\Big\|_{L^{4}(\RR^2)}.
\end{align*}
For $|x-2^{|k|+2N}e_1|\ge 2^{2N-1}$, we have
\begin{align*}
&\Big\|\frac{1}{(1+2^k|x-2^{|k|+2N}e_1|)}\frac{1}{(1+2^m|x-2^{|m|+2N}e_1|)}\Big\|_{L^{4}(\{|x-2^{|k|+2N}e_1|\ge 2^{2N-1}\})}\\
\le&C2^{-k}2^{-2N}\Big\|\frac{1}{(1+2^m|x-2^{|m|+2N}e_1|)}\Big\|_{L^{4}(\RR^2)}\le C2^{-k}2^{-2N}2^{-\frac{1}{2}m}.
\end{align*}
For $|x-2^{|k|+2N}e_1|\le 2^{2N-1}$, then we derive
\[|x-2^{|m|+2N}e_1|\ge|2^{|k|+2N}e_1-2^{|m|+2N}e_1|-|x-2^{|k|+2N}e_1|\ge2^{2N-1},\]
from which we obtain that
\begin{align*}
&\Big\|\frac{1}{(1+2^k|x-2^{|k|+2N}e_1|)}\frac{1}{(1+2^m|x-2^{|m|+2N}e_1|)}\Big\|_{L^{4}(\{|x-2^{|k|+2N}e_1|\le 2^{2N-1}\})}\\
\le&C2^{-m}2^{-2N}\Big\|\frac{1}{(1+2^k|x-2^{|k|+2N}e_1|)}\Big\|_{L^{4}(\RR^2)}\le C2^{-m}2^{-2N}2^{-\frac{1}{2}k}.
\end{align*}
Therefore, we conclude that
\begin{align*}
I_{j3}\le \frac{Ct_02^{2N}}{(\ln N)^2N^{\frac{1}{2}}}\sum_{-\delta N\le k\neq m\le 0}2^{-2N}2^{-\frac{1}{2}(k+m)}\le \frac{C2^{-2N}}{(\ln N)^3N^{\frac{1}{2}}}.
\end{align*}
Similarly, $I_{j4}$ can be bounded by
\begin{align*}
I_{j4}\le \frac{Ct_02^{2N}}{(\ln N)^2N^{\frac{1}{2}}}\sum_{-\delta N\le k\neq m\le 0}2^{-2N}2^{-\frac{1}{2}(k+m)}2^{-N}2^m\le \frac{C2^{-2N}}{(\ln N)^3N^{\frac{1}{2}}}.
\end{align*}
To sum up, we get that
\begin{equation}\label{Ij1}
\begin{aligned}
&\big(\sum_{-\delta N\le j\le 0}2^{-\frac{1}{2}jq}I^q_j\big)^{\frac{1}{q}}
\ge \frac{C}{(\ln N)^3N^{\frac{1}{2}}}(\sum_{-\delta N\le j\le 0}1)^{\frac{1}{q}}=\frac{CN^{\frac{1}{q}-\frac{1}{2}}}{(\ln N)^3}.
\end{aligned}
\end{equation}
Now we estimate the upper bound of $II_j$. Due to $t_0e^{-t_0|\xi|^2}-t_0=t_0\sum_{k=1}^\infty\frac{(-t_0 |\xi|)^k}{k!},$ then
\begin{align*}
II_j\le &\sum_{k=1}^\infty \frac{t^{k+1}_0}{k!} \big\|\mathscr{F}^{-1}\big((-|\xi|^2)^k\phi_j(\xi)
\int_{\RR^2}\widehat{u^1_0}(\xi-\eta)i\eta_1 \widehat{u^2_0}(\eta)\dd \eta\big)\big\|_{L^4}\\
=&\sum_{k=1}^\infty \frac{t^{k+1}_0}{k!} \big\|\big(\Delta^k\dot\Delta_j(u^1_0\partial_{x_1}u^2_0)
\big)\big\|_{L^4}\\
\le&\sum_{k=1}^\infty \frac{C^kt^{k+1}_0}{k!}2^{2jk}2^{\frac{j}{2}}2^N\|u^1_0\|_{L^4}\|u^2_0\|_{L^4}
\le\frac{C2^{\frac{5j}{2}}}{(\ln N)^4}.
\end{align*}
Therefore, it is easy to obtain that
\begin{align*}
\big(\sum_{-\delta N\le j\le 0}2^{-\frac{1}{2}jq}II^q_j\big)^{\frac{1}{q}}\le \big(\sum_{-\delta N\le j\le 0}\frac{C2^{2jq}}{(\ln N)^4}\big)^{\frac{1}{q}}\le \frac{C}{(\ln N)^4}.
\end{align*}
For $III_j$, it is easy to check that
\begin{equation}\label{G}
\begin{aligned}
G(\xi-\eta,\eta)
=&t_0e^{-t_0|\xi|^2}\sum_{k=1}^\infty\frac{(-1)^{k}t^{k}_0}{(k+1)!}\sum_{m=0}^{k-\ell}\sum_{\ell=0}^{k}C^\ell_kC^{m}_{k-\ell}|\xi-\eta|^{2\ell}|\eta|^{2m}(-|\xi|^2)^{k-\ell-m}.
\end{aligned}
\end{equation}
Noting the fact that supp $\widehat {u}_0\sim 2^N\mathcal{C}$ and $|\xi|\le 2$, then $G(\xi-\eta, \eta)=O((\ln N)^{-2}2^{-2N})$. Hence, by {Lemma} \ref{multi}, for $-\delta N\le j\le0$, we obtain that
\begin{equation}\label{IIIj}
\begin{aligned}
III_j\le&C\sum_{k=1}^\infty\frac{t^{k+1}_0}{(k+1)!}\sum_{m=0}^{k-\ell}\sum_{\ell=0}^{k}
\frac{k!}{(k-\ell-m)!}\|\dot\Delta_j e^{t_0\Delta}\Delta^{k-\ell-m}\big((-\Delta)^{\ell}u^{(1)}_0(-\Delta)^{m}\partial_{x_1}u^{(2)}_0\big)\|_{L^4}\\
\le&C\sum_{k=1}^\infty\frac{t^{k+1}_0}{(k+1)!}\sum_{m=0}^{k-1-\ell}\sum_{\ell=0}^{k}
\frac{k!2^{2j(k-\ell-m)}}{(k-\ell-m)!}\|(-\Delta)^{\ell}u^{(1)}_0\|_{L^8}\|(-\Delta)^{m}\partial_{x_1}u^{(2)}_0\|_{L^8}\\
&+C\sum_{k=1}^\infty\frac{t^{k+1}_0}{(k+1)!}\sum_{\ell=0}^{k}\|\dot\Delta_j(\Delta^{\ell}u^{(1)}_0\Delta^{k-\ell}\partial_{x_1}u^{(2)}_0)\|_{L^4}\\
\le&C\sum_{k=1}^\infty\frac{t^{k+1}_0}{(k+1)!}\sum_{m=0}^{k-1-\ell}\sum_{\ell=0}^{k}
\frac{k!2^{2j(k-\ell-m)}}{(k-\ell-m)!}\frac{2^{2N(\ell+m)}\cdot2^{2N}}{N^{\frac{1}{2}}(\ln N)^2}\\
&+C\sum_{k=1}^\infty\frac{t^{k+1}_0}{(k+1)!}\sum_{\ell=0}^{k}\|\dot\Delta_j(\Delta^{\ell}u^{(1)}_0\Delta^{k-\ell}\partial_{x_1}u^{(2)}_0)\|_{L^4}\\
\le&C2^{2j}\sum_{k=1}^\infty\frac{3^{k-1}t^{k+1}_0}{(k+1)!}\frac{2^{2N(k-1)}\cdot2^{2N}}{N^{\frac{1}{2}}(\ln N)^2}+C\sum_{k=1}^\infty\frac{t^{k+1}_0}{(k+1)!}\sum_{\ell=0}^{k}\|\dot\Delta_j(\Delta^{\ell}u^{(1)}_0\Delta^{k-\ell}\partial_{x_1}u^{(2)}_0)\|_{L^4}\\
\le&\frac{C2^{2j}2^{-2N}}{N^{\frac{1}{2}}(\ln N)^4}+C\sum_{k=1}^\infty\frac{t^{k+1}_0}{(k+1)!}\sum_{\ell=0}^{k}\|\dot\Delta_j(\Delta^{\ell}u^{(1)}_0\Delta^{k-\ell}\partial_{x_1}u^{(2)}_0)\|_{L^4}.
\end{aligned}
\end{equation}
For the last term on the above inequality,  it is enough to estimate
\begin{align*}
&\frac{2^N}{N^{\frac{1}{2}}(\ln N)^2}\Big\|\dot\Delta_j(\sum_{-\delta N\le m\le0}2^{\frac{m}{2}}\Phi_{m,N}(\Delta^{\ell}\sin(2^Nx_1))\sum_{-\delta N\le n\le0}2^{\frac{n}{2}}\Phi_{k,N}(\Delta^{k-\ell}\partial_{x_1}\cos(2^Nx_1)))\Big\|_{L^4}\\
\le&\frac{2^{2N(k+1)}}{N^{\frac{1}{2}}(\ln N)^2}\Big\|\dot\Delta_j(\sum_{-\delta N\le m\le0}2^{\frac{m}{2}}\Phi_{m,N}\sum_{-\delta N\le n\le0}2^{\frac{n}{2}}\Phi_{k,N}\sin^2(2^Nx_1))\Big\|_{L^4}\\
\le&\frac{2^{2N(k+1)}}{N^{\frac{1}{2}}(\ln N)^2}\Big(2^{j}\|\dot\Delta_j (\Phi^2_{j,N}\sin^2(2^Nx_1))\|_{L^4}+\sum_{-\delta N\le m\neq j\le0}2^m\|\dot\Delta_j (\Phi^2_{m,N}\sin^2(2^Nx_1))\|_{L^4}\\
&+\sum_{-\delta N\le m\neq n\le0}2^{\frac{m+n}{2}}\|\dot\Delta_j (\Phi_{m,N}\Phi_{n,N}\sin^2(2^Nx_1))\|_{L^4}\Big).
\end{align*}
By the definition of $\Phi_{j, N}$, one yields that
\begin{align*}
2^{j}\|\dot\Delta_j (\Phi^2_{j,N}\sin^2(2^Nx_1))\|_{L^4}\le C2^j\|\Phi_{j,N}\|^2_{L^8}\le C2^{\frac{1}{2}j}.
\end{align*}
Following the methods on $I_{j2}$ and $I_{j3}$, it is easy to get that
\begin{align*}
&\sum_{-\delta N\le m\neq j\le0}2^m\|\dot\Delta_j (\Phi^2_{m,N}\sin^2(2^Nx_1))\|_{L^4}+\sum_{-\delta N\le m\neq n\le0}2^{\frac{m+n}{2}}\|\dot\Delta_j (\Phi_{m,N}\Phi_{n,N}\sin^2(2^Nx_1))\|_{L^4}\\
\le&C2^{-2N}2^{-\frac{1}{2}j}+ C2^{-2N}.
\end{align*}
Therefore, we obtain that
\begin{align*}
&\frac{2^N}{N^{\frac{1}{2}}(\ln N)^2}\Big\|\dot\Delta_j(\sum_{-\delta N\le m\le0}2^{\frac{m}{2}}\Phi_{m,N}(\Delta^{\ell}\sin(2^Nx_1))\sum_{-\delta N\le n\le0}2^{\frac{n}{2}}\Phi_{k,N}(\Delta^{k-\ell}\partial_{x_1}\cos(2^Nx_1)))\Big\|_{L^4}\\
\le&\frac{C2^{2N(k+1)}}{N^{\frac{1}{2}}(\ln N)^2}(2^{\frac{1}{2}j}+2^{-2N}2^{-\frac{1}{2}j}+ 2^{-2N})\le \frac{C2^{2N(k+1)}2^{\frac{j}{2}}}{N^{\frac{1}{2}}(\ln N)^2}.
\end{align*}
This implies that
\begin{align*}
III_j\le \frac{C2^{2j}2^{-2N}}{N^{\frac{1}{2}}(\ln N)^4}+C2^{\frac{j}{2}}\sum_{k=1}^\infty\frac{k(t_02^{2N})^{k+1}}{(k+1)!}\le\frac{C(2^{2j}2^{-2N}+2^{{\frac{1}{2}}j})}{N^{\frac{1}{2}}(\ln N)^4}.
\end{align*}
Therefore, we have that
\begin{align*}
\big(\sum_{-\delta N\le j\le 0}2^{-\frac{1}{2}jq}III^q_j\big)^{\frac{1}{q}}\le \frac{CN^{\frac{1}{q}-\frac{1}{2}}}{(\ln N)^4}.
\end{align*}
We  utilize estimates on $I_j$ and $II_j$ to see that
\begin{align*}
\big(\sum_{-\delta N\le j\le 0}2^{-\frac{1}{2}jq}\big\|\dot\Delta_j\int_0^{t_0} e^{(t_0-s)\Delta}(U^1_0\partial_{x_1} U^2_0)\dd s\big\|^q_{L^4}\big)^{\frac{1}{q}}\ge \frac{C {N^{\frac{1}{q}-\frac{1}{2}}}}{(\ln N)^3}-\frac{C}{(\ln N)^4}-\frac{C {N^{\frac{1}{q}-\frac{1}{2}}}}{(\ln N)^4}\ge \frac{C{N^{\frac{1}{q}-\frac{1}{2}}}}{(\ln N)^3}.
\end{align*}
Now we turn to estimate the third term on the right-hand side of inequality \eqref{U_1}. An easy computation yields that
\begin{align*}
&\Big(\sum_{-\delta N\le j\le 0}2^{-\frac{1}{2}jq}\big{\|}\dot\Delta_j\int_0^t e^{(t-s)\Delta}\nabla \int_0^s \Div U_0(\tau)\dd \tau\cdot\nabla U^{(2)}_0(s)\dd s\big{\|}^q_{L^4}\Big)^{\frac{1}{q}}\\
\le&\Big(\sum_{-\delta N\le j\le 0}2^{-\frac{1}{2}jq}\big{\|}\dot\Delta_j\int_0^t e^{(t-s)\Delta}\big( \big(\int_0^s \partial^2_{x_1}U^{(1)}_0(\tau)\dd \tau\big)\partial_{x_1} U^{(2)}_0(s)\big)\dd s\big{\|}^q_{L^4}\Big)^{\frac{1}{q}}\\
&+\Big(\sum_{-\delta N\le j\le 0}2^{-\frac{1}{2}jq}\big{\|}\dot\Delta_j\int_0^t e^{(t-s)\Delta}\big( \big(\int_0^s \partial_{x_1}\partial_{x_2}U^{(2)}_0(\tau)\dd \tau\big)\partial_{x_1} U^{(2)}_0(s)\big)\dd s\big{\|}^q_{L^4}\Big)^{\frac{1}{q}}\\
&+\Big(\sum_{-\delta N\le j\le 0}2^{-\frac{1}{2}jq}\big{\|}\dot\Delta_j\int_0^t e^{(t-s)\Delta}\big( \big(\int_0^s \partial_{x_1}\partial_{x_2}U^{(1)}_0(\tau)\dd \tau\big)\partial_{x_2} U^{(2)}_0(s)\big)\dd s\big{\|}^q_{L^4}\Big)^{\frac{1}{q}}\\
&+\Big(\sum_{-\delta N\le j\le 0}2^{-\frac{1}{2}jq}\big{\|}\dot\Delta_j\int_0^t e^{(t-s)\Delta}\big( \big(\int_0^s \partial^2_{x_2}U^{(2)}_0(\tau)\dd \tau\big)\partial_{x_2} U^{(2)}_0(s)\Big)\dd s\big{\|}^q_{L^4}\Big)^{\frac{1}{q}}.
\end{align*}
For the last three terms on the right-hand side of the above inequality, we obtain that
\begin{align*}
&\Big(\sum_{-\delta N\le j\le 0}2^{-\frac{1}{2}jq}\big{\|}\dot\Delta_j\int_0^t e^{(t-s)\Delta}\big( \big(\int_0^s \partial_{x_1}\partial_{x_2}U^{(2)}_0(\tau)\dd \tau\big)\partial_{x_1} U^{(2)}_0(s)\big)\dd s\big{\|}^q_{L^4}\Big)^{\frac{1}{q}}\\
&+\Big(\sum_{-\delta N\le j\le 0}2^{-\frac{1}{2}jq}\big{\|}\dot\Delta_j\int_0^t e^{(t-s)\Delta}\big( \big(\int_0^s \partial_{x_1}\partial_{x_2}U^{(1)}_0(\tau)\dd \tau\big)\partial_{x_2} U^{(2)}_0(s)\big)\dd s\big{\|}^q_{L^4}\Big)^{\frac{1}{q}}\\
&+\Big(\sum_{-\delta N\le j\le 0}2^{-\frac{1}{2}jq}\big{\|}\dot\Delta_j\int_0^t e^{(t-s)\Delta}\big( \big(\int_0^s \partial^2_{x_2}U^{(2)}_0(\tau)\dd \tau\big)\partial_{x_2} U^{(2)}_0(s)\big)\dd s\big{\|}^q_{L^4}\Big)^{\frac{1}{q}}\\
\le&CN^{\frac{1}{q}}T^2(\|\partial_{x_1x_2}u^{(2)}_0\|_{L^4}\|\partial_{x_1}u^{(2)}_0\|_{L^4}
+\|\partial_{x_1x_2}u^{(1)}_0\|_{L^4}\|\partial_{x_2}u^{(2)}_0\|_{L^4}
+\|\partial^2_{x_2}u^{(2)}_0\|_{L^4}\|\partial_{x_2}u^{(2)}_0\|_{L^4})\\
\le& \frac{CN^{\frac{1}{q}}2^{-N}}{(\ln N)^4}.
\end{align*}
By Fourier transform, we get that
\begin{align*}
&\mathscr{F}\Big(\int_0^t e^{(t-s)\Delta}\Big( \big(\int_0^s \partial^2_{x_1}U^{(1)}_0(\tau)\dd \tau\big)\partial_{x_1} U^{(2)}_0(s)\Big)\dd s\Big)\\
=&-\int_0^te^{-(t-s)|\xi|^2}\int_{\RR^2}\int_0^s|\xi_1-\eta_1|^2e^{-\tau|\xi-\eta|^2}\widehat{u^{(2)}_0}(\xi-\eta)
\dd\tau i\eta_1 e^{-s|\eta|^2}\widehat{u^{(2)}_0}(\eta)\dd\eta\dd s\\
=&\int_{\RR^2}e^{-t|\xi|^2}\Big(\frac{e^{t(|\xi|^2-|\eta|^2-|\xi-\eta|^2)}-1}{|\xi|^2-|\eta|^2-|\xi-\eta|^2}
-\frac{e^{t(|\xi|^2-|\eta|^2)}-1}{|\xi|^2-|\eta|^2}\Big)\frac{|\xi_1-\eta_1|^2}{|\xi-\eta|^2}\widehat{u^{(2)}_0}(\xi-\eta)
 i\eta_1 \widehat{u^{(2)}_0}(\eta)\dd\eta.
\end{align*}
Using  Taylor's series, it easily yields that
\begin{align*}
&\frac{e^{t(|\xi|^2-|\eta|^2-|\xi-\eta|^2)}-1}{|\xi|^2-|\eta|^2-|\xi-\eta|^2}
-\frac{e^{t(|\xi|^2-|\eta|^2)}-1}{|\xi|^2-|\eta|^2}\\
=&\sum_{k=1}^\infty\frac{t^{k+1}(|\xi|^2-|\eta|^2-|\xi-\eta|^2)^{k}}{(k+1)!}
-\sum_{k=1}^\infty\frac{t^{k+1}(|\xi|^2-|\eta|^2)^{k}}{(k+1)!}.
\end{align*}
Therefore, by \eqref{G} and estimates on $III_j$,  we have
\begin{align*}
&\Big{\|}\dot\Delta_j\int_0^t e^{(t-s)\Delta}\Big( \big(\int_0^s \partial^2_{x_1}U^{(1)}_0(\tau)\dd \tau\big)\partial_{x_1} U^{(2)}_0(s)\Big)\dd s\Big{\|}_{L^4}\\
\le&\Big{\|}\mathscr{F}^{-1}\Big(\widehat{\varphi}_j(\xi)\int_{\RR^2}e^{-t|\xi|^2}\Big(\sum_{k=1}^\infty\frac{t^{k+1}(|\xi|^2-|\eta|^2-|\xi-\eta|^2)^{k}}{(k+1)!}\Big)\frac{|\xi_1-\eta_1|^2}{|\xi-\eta|^2}\widehat{u^{(2)}_0}(\xi-\eta)
 i\eta_1 \widehat{u^{(2)}_0}(\eta)\dd\eta\Big)\Big{\|}_{L^{4}}\\
 &+\Big{\|}\mathscr{F}^{-1}\Big(\widehat{\varphi}_j(\xi)\int_{\RR^2}e^{-t|\xi|^2}
 \Big(\sum_{k=1}^\infty\frac{t^{k+1}(|\xi|^2-|\eta|^2)^{k}}{(k+1)!}\Big)\frac{|\xi_1-\eta_1|^2}{|\xi-\eta|^2}\widehat{u^{(2)}_0}(\xi-\eta)
 i\eta_1 \widehat{u^{(2)}_0}(\eta)\dd\eta\Big)\Big{\|}_{L^{4}}\\
 \le&\sum_{k=1}^\infty\frac{t_0^{k+1}}{(k+1)!}\sum_{m=0}^{k-\ell}\sum_{\ell=0}^{k}C^\ell_kC^m_{k-\ell}\|\dot\Delta_je^{t_0\Delta}(-\Delta)^{k-\ell-m}\big(\Delta^{\ell-1}{\partial^2_{x_1}u^{(2)}_0}\Delta^m\partial_{x_1}u^{(2)}_0\big)\|_{L^4}\\
 &+\sum_{k=1}^\infty\frac{t_0^{k+1}}{(k+1)!}\sum_{\ell=0}^{k}C^\ell_k\|\dot\Delta_je^{t_0\Delta}(-\Delta)^{k-\ell}\big(\Delta^{-1}{\partial^2_{x_1}u^{(2)}_0}\Delta^\ell\partial_{x_1}u^{(2)}_0\big)\|_{L^4}
 \le\frac{C2^{\frac{1}{2}j}}{N^{\frac{1}{2}}(\ln N)^4}.
\end{align*}
Hence, we obtain that
\begin{align*}
\Big(\sum_{-\delta N\le j\le 0}2^{-\frac{1}{2}jq}\Big{\|}\dot\Delta_j\int_0^t e^{(t-s)\Delta}\big(\nabla \int_0^s \Div U_0(\tau)\dd \tau\cdot\nabla U^{(2)}_0(s)\big)\dd s\Big{\|}^q_{L^4}\Big)^{\frac{1}{q}}
\le\frac{ CN^{\frac{1}{q}-\frac{1}{2}}}{(\ln N)^4}+\frac{CN^{\frac{1}{q}}2^{-N}}{(\ln N)^4}.
\end{align*}
To sum up,
\begin{equation}\label{U__1}
\|U_1(t_0)\|_{\dot B^{-\frac{1}{2}}_{4, 1}}\ge \frac{CN^{\frac{1}{q}-\frac{1}{2}}}{(\ln N)^3}-C(\ln N)^{-2}-\frac{ CN^{\frac{1}{q}-\frac{1}{2}}}{(\ln N)^4}-\frac{CN^{\frac{1}{q}}2^{-N}}{(\ln N)^4}\ge\frac{CN^{\frac{1}{q}-\frac{1}{2}}}{(\ln N)^3}.
\end{equation}
Now we need to estimate $\|U_2(t_0)\|_{\dot B^0_{[2,\infty],1}}$.
Based on Proposition \ref{minius}, we obtain that
\begin{align*}
&\|u-U_0-U_1\|_{L^\infty_T(\dot B^{-\frac{1}{2}}_{4, 1})}\\
\le &C(\|(u-U_0)\cdot\nabla u\|_{L^1_T (\dot B^{-\frac{1}{2}}_{4, 1})}+\|U_0\cdot\nabla (u-U_0)\|_{L^1_T (\dot B^{-\frac{1}{2}}_{4, 1})}+\|\nabla(\ln(1+h)-h)\cdot\nabla u\|_{L^1_T(\dot B^{-\frac{1}{2}}_{4, 1})}\\
&+T\|h\|_{L^\infty_T (\dot B^{\frac{3}{4}}_{4, 1})}+\|\int_0^te^{(t-s)\Delta}\nabla(h+\int_0^t \Div U_0)\cdot\nabla u\|_{L^\infty_T(\dot B^{-\frac{1}{2}}_{4, 1})}\\
&+\|\nabla\int_0^t \Div U_0\cdot\nabla (u-U_0)\|_{L^1_T(\dot B^{-\frac{1}{2}}_{4, 1})}.
\end{align*}
Actually, by inequality \eqref{u-u0}, we have
\begin{align*}
&\|(u-U_0)\cdot\nabla u\|_{L^1_T (\dot B^0_{[2,\infty],1})}+\|U_0\cdot\nabla (u-U_0)\|_{L^1_T (\dot B^0_{[2,\infty],1})}\\
\le&CT^{\frac{1-\e}{2}}\|u-U_0\|_{L^\infty_T(\B)}\|u\|_{L^{\frac{2}{1+\e}}_T(\dot B^{1+\e}_{[2,\infty],1})}+CT^{\frac{1-\e}{2}}\|u_0\|_{\B}\|u-U_0\|_{L^{\frac{2}{1+\e}}_T(\dot B^{1+\e}_{[2,\infty],1})}\\
\le&C2^{(-\frac{1}{2}+\frac{3\delta}{2}+2\e)N}.
\end{align*}
By bilinear estimates, it follows that
\begin{align*}
&\|\nabla(\ln(1+h)-h)\cdot\nabla u\|_{L^1_T (\B)}=\Big\|\frac{h}{1+h}\nabla h\cdot\nabla u\Big\|_{L^1_T (\B)}\\
\le&CT^{\frac{1-\e}{2}}\Big\|\frac{h}{1+h}\Big\|_{L^\infty_T(\dot B^{\e}_{[2,\infty],1})}\|h\|_{L^\infty_T(\Bo)}\|u\|_{L^{\frac{2}{1+\e}}_T(\dot B^{1+\e}_{[2,\infty],1})}\\
\le&CT^{\frac{1-\e}{2}}\sum_{m=1}^\infty C^m\|h\|^{m}_{L^\infty_T(\B)}\|h\|_{L^\infty_T(\dot B^{\e}_{[2,\infty],1})}\|h\|_{L^\infty_T(\Bo)}\|u\|_{L^{\frac{2}{1+\e}}_T(\dot B^{1+\e}_{[2,\infty],1})}\\
\le&C2^{(-\frac{1}{2}+\e+\frac{\delta}{2})N}\cdot 2^{(1+{\delta})N}\cdot 2^{(-1+\e)N}=C2^{(-\frac{1}{2}+2\e+\frac{3\delta}{2})N}.
\end{align*}
Moreover, it holds that
\begin{align*}
&\|\int_0^te^{(t-s)\Delta}\big(\nabla(h+\int_0^s \Div U_0\dd\tau)\cdot\nabla u\big)\dd s\|_{L^\infty_T(\dot B^{-\frac{1}{2}}_{4, 1})}\\
\le&\|\int_0^te^{(t-s)\Delta}\Div ((h+\int_0^s \Div U_0\dd\tau)\nabla u)\dd s\|_{L^\infty_T(\dot B^{-\frac{1}{2}}_{4, 1})}\\
&+\|\int_0^te^{(t-s)\Delta}((h+\int_0^s \Div U_0\dd \tau)\Delta u)\dd s\|_{L^\infty_T(\dot B^{-\frac{1}{2}}_{4, 1})}\\
\le&C\|(h+\int_0^t \Div U_0\dd \tau)\nabla u(t)\|_{L^{\frac{4}{3}}_T (\dot B^{0}_{4,1})}+\|(h+\int_0^t \Div U_0)\Delta u(t)\|_{L^1_T(\dot B^{-\frac{1}{2}}_{4,1})}\\
\le&C\|h+\int_0^t \Div U_0\dd \tau\|_{L^\infty_T(\dot B^{\e}_{[2,\infty],1})}(T^{\frac{1}{4}}\|u\|_{L^2_T( \Bo)}+\|u\|_{L^2_T(\Bt)})\\
\le&C2^{(\frac{3\delta}{2}+\e(\frac{5}{2}-\frac{\delta}{2}-\e)-\frac{1}{2})N},
\end{align*}
and
\begin{align*}
\|\nabla\int_0^t \Div U_0\dd s\cdot\nabla (u-U_0)\|_{L^1_T\B}
\le&C\|\nabla\int_0^t \Div U_0\dd s\|_{L^\infty_T\B}\|\nabla (u-U_0)\|_{L^1_T\dot B^{\e}_{[2,\infty],1}}\\
\le&CT^{\frac{1-\e}{2}}\|U_0\|_{L^1_T\Bt}\|u-U_0\|_{L^{\frac{2}{1+\e}}_T\dot B^{1+\e}_{[2,\infty],1}}\\
\le&C2^{(-\frac{1}{2}+2\e+\frac{3}{2}\delta)N}.
\end{align*}
Noting the fact that $3\delta+5\e<1$, we obtain that
\begin{equation}\label{U2}
\|U_2(t_0)\|_{\dot B^0_{[2,\infty],1}}\le C2^{(-\frac{1}{2}+\frac{5}{2}\e+\frac{3}{2}\delta)N}\le C.
\end{equation}
To sum up, combined with \eqref{U_0}, \eqref{U__1} and \eqref{U2}, we get that for $1\le q<2$, $t_0=(\ln N)^{-1}2^{-2N}$,
\begin{align*}
\|u(t_0)\|_{\dot B^{-\frac{1}{2}}_{4,q}}&\ge \|U_1(t_0)\|_{\dot B^{-\frac{1}{2}}_{4,q}}-\|U_0(t_0)\|_{\dot B^{-\frac{1}{2}}_{4,q}}-\|U_2(t_0)\|_{\dot B^{-\frac{1}{2}}_{4,q}}\\
&\ge C(\ln N)^{-3}N^{\frac{1}{q}-\frac{1}{2}}-C{(\ln N)^{-1}}-C\\
&\ge C(\ln N)^{-3}N^{\frac{1}{q}-\frac{1}{2}},
\end{align*}
with initial data $\|u_0\|_{\dot B^{-\frac{1}{2}}_{4,q}}\le C(\ln N)^{-1}$.
\subsection{Proof of Theorem \ref{thm} for \bm{$q>2$}.}
For $q>2$, we construct initial data $(h_0, u_0)$  as follows:
\[h_0=0,\quad u_0=\Big(\frac{ \ln N}{\sqrt{N}}\sum_{ N\le k\le (1+\delta)N}2^{\frac{k}{2}}\varphi_0(x)\sin (2^k x_1), \frac{ \ln N}{\sqrt{N}}\sum_{N\le k\le (1+\delta)N}2^{\frac{k}{2}}\varphi_0(x)\cos (2^k x_1)\Big),\]
where $\varphi_0$ is the function stemming from localization homogeneous operator $\dot\Delta_0$ and $\delta$ is defined in Proposition \ref{exist}. For convenience, we set $\varphi_0(x)=\mathscr{F}^{-1}(\phi(\xi))$. It is easy to check that
\begin{align*}
\|u_0\|_{\dot B^{-\frac{1}{2}}_{4,q}}\le&\frac{ C\ln N}{\sqrt{N}} (\sum_{N-1\le j\le (1+\delta)N+1}\sum_{|j-k|\le 1}2^{\frac{1}{2}(j-k)q}\|\dot\Delta_j(\varphi_0(x)\sin(2^kx_1))\|^q_{L^4})^{\frac{1}{q}}
\le \frac{C\ln N}{ N^{\frac{1}{2}-\frac{1}{q}}}.
\end{align*}
Now we decompose the solution $u$ into three parts:
\[u=U_0+U_1+U_2,\]
here the decomposition is the same with \eqref{decomposition}, and next we respectively estimate $\|U_i(t_0)\|_{\dot B^{-\frac{1}{2}}_{4, q}}(i=0,1,2)$ for $t_0=(\ln N)^{-1}2^{-2N}$.

\noindent\emph{{Estimates on $\|U_0 (t_0)\|_{\dot B^{-\frac{1}{2}}_{4, 1}}$}}.\,\,
From estimates on initial data $u_0$, one yields that
\begin{align*}
\|U_0(t_0)\|_{\dot B^{-\frac{1}{2}}_{4, q}}\le C\|u_0\|_{\dot B^{-\frac{1}{2}}_{4, q}}\le \frac{C\ln N}{ N^{\frac{1}{2}-\frac{1}{q}}}.
\end{align*}
We denote the $k-th$ component of $U_1$ by $U^{(k)}_1$. Then
\begin{align*}
U^{(2)}_1=&-\int_0^t e^{(t-s)\Delta}(U_0\cdot\nabla U^{(2)}_0+\nabla \int_0^s \Div U_0(\tau)\dd \tau\cdot\nabla U^{(2)}_0(s))\dd s\\
=&-\int_0^t e^{(t-s)\Delta}(U^{(1)}_0\partial_{x_1}U^{(2)}_0+\int_0^s\partial^2_{x_1}U^{(1)}_0\dd \tau\partial_{x_1}U^{(2)}_0(s))\dd s\\
&-\int_0^t e^{(t-s)\Delta}(U^{(2)}_0\partial_{x_2}U^{(2)}_0+\int_0^s\partial_{x_1x_2}U^{(2)}_0\dd \tau\partial_{x_1}U^{(2)}_0(s))\dd s\\
&-\int_0^t e^{(t-s)\Delta}(\int_0^s\partial_{x_2}(\partial_{x_1}U^{(1)}_0+\partial_{x_2}U^{(2)}_0)\dd \tau\partial_{x_2}U^{(2)}_0(s))\dd s.
\end{align*}
\emph{{Estimates on $\|U_1(t_0)\|_{\dot B^{-\frac{1}{2}}_{4, q}}$}}.\,\,  Taking advantage of embedding $\dot B^{-\frac{1}{2}}_{4,q}(\RR^2)\hookrightarrow \dot B^{-1}_{\infty,\infty}(\RR^2)$, we can get that
\begin{align*}
&\|U_1(t_0)\|_{\dot B^{-\frac{1}{2}}_{4,q}}\ge C\|U^{(2)}_1(t_0)\|_{\dot B^{-1}_{\infty,\infty}}\ge C\|\dot\Delta_0 U_1(t_0)\|_{L^\infty}\\
\ge& C\Big\|\dot\Delta_0 \int_0^{t_0} e^{({t_0} -s)\Delta}(U^{(1)}_0\partial_{x_1}U^{(2)}_0+\int_0^s\partial^2_{x_1}U^{(1)}_0\dd \tau\partial_{x_1}U^{(2)}_0(s))\dd s\Big\|_{L^\infty}\\
&-C\Big\|\dot\Delta_0\int_0^{t_0}  e^{({t_0} -s)\Delta}(U^{(2)}_0\partial_{x_2}U^{(2)}_0+\int_0^s\partial_{x_1x_2}U^{(2)}_0\dd \tau\partial_{x_1}U^{(2)}_0(s))\dd s\Big\|_{L^\infty}\\
&-C\Big\|\dot\Delta_0\int_0^{t_0}  e^{({t_0} -s)\Delta}(\int_0^s\partial_{x_2}(\partial_{x_1}U^{(1)}_0+\partial_{x_2}U^{(2)}_0)\dd \tau\partial_{x_2}U^{(2)}_0(s))\dd s\Big\|_{L^\infty}.
\end{align*}
By $\|f\|_{L^\infty}\ge|f(0)|=|\int_{\RR^2} \widehat{f}(\xi)\dd \xi|$, we obtain that
\begin{equation}\label{U1}
\begin{aligned}
\|U_1{(t_0)} \|_{\dot B^{-\frac{1}{2}}_{4,q}}
\ge& C\Big|\int_{\RR^2}\phi(\xi) \mathscr{F}\Big(\int_0^{t_0}  e^{({t_0} -s)\Delta}(U^{(1)}_0\partial_{x_1}U^{(2)}_0+\int_0^s\partial^2_{x_1}U^{(1)}_0\dd \tau\partial_{x_1}U^{(2)}_0(s))\dd s\Big)\dd \xi\Big|\\
&-C\Big\|\dot\Delta_0\int_0^{t_0}  e^{({t_0} -s)\Delta}(U^{(2)}_0\partial_{x_2}U^{(2)}_0+\int_0^s\partial_{x_1x_2}U^{(2)}_0\dd \tau\partial_{x_1}U^{(2)}_0(s))\dd s\Big\|_{L^\infty}\\
&-C\Big\|\dot\Delta_0\int_0^{t_0}  e^{({t_0} -s)\Delta}(\int_0^s\partial_{x_2}(\partial_{x_1}U^{(1)}_0+\partial_{x_2}U^{(2)}_0)\dd \tau\partial_{x_2}U^{(2)}_0(s))\dd s\Big\|_{L^\infty}.
\end{aligned}
\end{equation}
According to the definition of $U_0$, one yields that
\begin{align*}
&\Big|\int_{\RR^2}\phi(\xi)\mathscr{F}\Big(\int_0^{t_0}  e^{({t_0} -s)\Delta}(U^{(1)}_0\partial_{x_1}U^{(2)}_0+\int_0^s\partial^2_{x_1}U^{(1)}_0\dd \tau\partial_{x_1}U^{(2)}_0(s))\dd s\Big)\dd \xi\Big|\\
=&\Big|\int_{\RR^2}\phi(\xi)\mathscr{F}\Big(\int_0^{t_0}  e^{({t_0} -s)\Delta}(U^{(1)}_0+\int_0^s\partial^2_{x_1}U^{(1)}_0\dd \tau)\partial_{x_1}U^{(2)}_0(s))\dd s\Big)\dd \xi\Big|\\
=&\Big|\int_{\RR^2}\phi(\xi) \Big(\int_0^{t_0}  e^{-({t_0} -s)|\xi|^2}\int_{\RR^2}\big(e^{-s|\xi-\eta|^2}\widehat{u^1_0}(\xi-\eta)-(\xi_1-\eta_1)^2\int_0^s e^{-\tau|\xi-\eta|^2}\widehat{u^1_0}(\xi-\eta)\dd \tau\big)\\
&\times i\eta_1e^{-s|\eta|^2}\widehat{u^2_0}(\eta)\dd\eta\dd s\Big)\dd \xi\Big|\\
=&\Big|\int_{\RR^2}\phi(\xi) \Big(\int_0^{t_0}  e^{-({t_0} -s)|\xi|^2}\int_{\RR^2}\big(e^{-s|\xi-\eta|^2}\widehat{u^1_0}(\xi-\eta)-\frac{|\xi_1-\eta_1|^2}{|\xi-\eta|^2}(1-e^{-s|\xi-\eta|^2})\widehat{u^1_0}(\xi-\eta)\big)\\
&\times i\eta_1e^{-s|\eta|^2}\widehat{u^2_0}(\eta)\dd\eta\dd s\Big)\dd \xi\Big|\\
\ge&\Big|\int_{\RR^2}\phi(\xi) \Big(\int_0^{t_0}  e^{-({t_0} -s)|\xi|^2}\int_{\RR^2}\big(2e^{-s|\xi-\eta|^2}-1\big)\widehat{u^1_0}(\xi-\eta)i\eta_1e^{-s|\eta|^2}\widehat{u^2_0}(\eta)\dd\eta\dd s\Big)\dd \xi\Big|\\
&-\Big|\int_{\RR^2}\phi(\xi) \Big(\int_0^{t_0}  e^{-({t_0} -s)|\xi|^2}\int_{\RR^2}\frac{|\xi_2-\eta_2|^2}{|\xi-\eta|^2}(1-e^{-s|\xi-\eta|^2})\widehat{u^1_0}(\xi-\eta)i\eta_1e^{-s|\eta|^2}\widehat{u^2_0}(\eta)\dd\eta\dd s\Big)\dd \xi\Big|\\
:=&I-II.
\end{align*}
By Fourier transform, we obtain that
\begin{equation}\label{u01u02}
\begin{aligned}
&\widehat{u^{1}_0}=\frac{\ln N}{2\sqrt{N}}\sum_{N\le k\le(1+\delta)N}2^{\frac{k}{2}}(i\phi(\xi+2^{k}e_1)-i\phi(\xi-2^{k}e_1)),\\
&\widehat{u^{2}_0}=\frac{\ln N}{2\sqrt{N}}\sum_{N\le k\le(1+\delta)N}2^{\frac{k}{2}}(\phi(\xi+2^{k}e_1)+\phi(\xi-2^{k}e_1)).
\end{aligned}
\end{equation}
Noting the fact that supp $\phi(\xi) =\{\xi\in \RR^2| \frac{3}{4}\le|\xi|\le\frac{8}{3}\}$, therefore
\begin{align*}
I=&\frac{(\ln N)^2}{4N}\Big|\sum_{N\le k\le(1+\delta)N}\int_{\RR^2}\phi(\xi) \Big(\int_0^{t_0}  e^{-({t_0} -s)|\xi|^2}\int_{\RR^2}\big(2e^{-s|\xi-\eta|^2}-1\big)e^{-s|\eta|^2}\\
&\times2^k(\eta_1\phi(\xi-\eta+2^{k}e_1)\phi(\eta-2^{k}e_1)-\eta_1\phi(\xi-\eta-2^{k}e_1)
\phi(\eta+2^{k}e_1))\dd\eta\dd s\dd\xi\Big|.
\end{align*}
By triangle inequality, we have
\begin{align*}
I\ge&\frac{(\ln N)^2}{4N}\Big|\sum_{N+\log_4(\ln 4\cdot\ln N)\le k\le(1+\delta)N}\int_{\RR^2}\phi(\xi) \Big(\int_0^{t_0}  e^{-({t_0} -s)|\xi|^2}\int_{\RR^2}\big(2e^{-s|\xi-\eta|^2}-1\big)e^{-s|\eta|^2}\\
&\times2^k(\eta_1\phi(\xi-\eta+2^{k}e_1)\phi(\eta-2^{k}e_1)-\eta_1\phi(\xi-\eta-2^{k}e_1)
\widehat{\varphi_0}(\eta+2^{k}e_1))\dd\eta\dd s\dd\xi\Big|\\
&-\frac{(\ln N)^2}{4N}\Big|\sum_{N\le k\le N+\log_4(\ln 4\cdot\ln N)}\int_{\RR^2}\phi(\xi) \Big(\int_0^{t_0}  e^{-({t_0} -s)|\xi|^2}\int_{\RR^2}\big(2e^{-s|\xi-\eta|^2}-1\big)e^{-s|\eta|^2}\\
&\times2^k(\eta_1\phi(\xi-\eta+2^{k}e_1)\phi(\eta-2^{k}e_1)-\eta_1\phi(\xi-\eta-2^{k}e_1)
\phi(\eta+2^{k}e_1))\dd\eta\dd s\dd\xi\Big|\\
:=&I_1-I_2.
\end{align*}
Noting the support of $\phi(\eta\pm2^{k}e_1)$ and $\phi(\xi-\eta\pm2^{k}e_1)$, we have $$\eta_1\phi(\eta-2^{k}e_1)\ge 0, \quad -\eta_1\phi(\eta+2^{k}e_1)\ge 0.$$
Because $t_0=(\ln N)^{-1}2^{-2N}$, it is easy to check that if $|\xi-\eta|\ge\frac{9}{10}2^k$, for any $k\ge N+\log_4(\ln 4\cdot\ln N)(N\gg 1)$,
\[2e^{-t|\xi-\eta|^2}\le 2e^{-(\ln N)^{-1}2^{-2N}\cdot 2^{2k}}=2e^{-(\frac{9}{10})^2(\ln N)^{-1}2^{-2N}\cdot 2^{2N}\cdot \ln4\cdot\ln N }=2e^{-(\frac{9}{10})^2\ln 4}\le \frac{3}{4}.\]
Therefore, due to $k\ge N\gg1$, it is easy to check that, $|\eta|\ge\frac{9}{10}2^k$ for $\eta\in \text{supp } \phi(\eta-2^k e_1)$,
\begin{align*}
I_1&\ge \frac{(\ln N)^2}{16N}\sum_{k=N+\log_4(\ln 4\cdot\ln N)}^{(1+\delta)N}\int_{\RR^2}\phi(\xi) \int_0^{t_0}  e^{-({t_0} -s)|\xi|^2}\int_{\RR^2}e^{-s|\eta|^2}\\
&\qquad\qquad\qquad\qquad\qquad\times2^k\eta_1\phi(\xi-\eta+2^{k}e_1)\phi(\eta-2^{k}e_1)\dd\eta\dd s\dd\xi\\
&=\frac{(\ln N)^2}{16N}\sum_{k=N+\log_4(\ln 4\cdot\ln N)}^{(1+\delta)N}\iint\phi(\xi) \frac{e^{-t|\xi|^2}-e^{-t|\eta|^2}}{|\eta|^2-|\xi|^2}2^k\eta_1\phi(\xi-\eta+2^{k}e_1)\phi(\eta-2^{k}e_1)\dd\eta\dd\xi\\
&\ge\frac{C(\ln N)^2}{N}\sum_{k=N+\log_4(\ln 4\cdot\ln N)}^{(1+\delta)N}\iint\phi(\xi) \phi(\xi-\eta+2^{k}e_1)\phi(\eta-2^{k}e_1)\dd\eta\dd\xi\\
&\ge C\delta(\ln N)^2.
\end{align*}
For $I_2$,  from above analysis, it yields that
\begin{align*}
I_2\le&\frac{C(\ln N)^2}{N}\sum_{k=N}^{N+\log_4(\ln 4\cdot\ln N)}\iint\phi(\xi) \phi(\xi-\eta+2^{k}e_1)\phi(\eta-2^{k}e_1)\dd\eta\dd\xi
\le\frac{C(\ln N)^3}{N}.
\end{align*}
Taking advantage of \eqref{u01u02}, $II$ can be written as
\begin{align*}
II&=\frac{(\ln N)^2}{N}\Big|\sum_{N\le k\le(1+\delta)N}\int_{\RR^2}\phi(\xi) \Big(\int_0^{t_0} e^{-(t_0t-s)|\xi|^2}\int_{\RR^2}\frac{|\xi_2-\eta_2|^2}{|\xi-\eta|^2}(1-e^{-s|\xi-\eta|^2})e^{-s|\eta|^2}\\
&\times2^k(\eta_1{\phi}(\xi-\eta+2^{k}e_1)\widehat{\phi}(\eta-2^{k}e_1)-\eta_1{\phi}(\xi-\eta-2^{k}e_1)
{\phi}(\eta+2^{k}e_1))\dd\eta\dd s\Big)\dd \xi\Big|\\
\le&\frac{(\ln N)^2}{N}\sum_{N\le k\le(1+\delta)N}\iint\phi(\xi)\frac{e^{-t_0|\xi|^2}-e^{-t_0|\eta|^2}}{|\eta|^2-|\xi|^2}\frac{|\xi_2-\eta_2|^2}{|\xi-\eta|^2}\\
&\times2^k(\eta_1{\phi}(\xi-\eta+2^{k}e_1){\phi}(\eta-2^{k}e_1)-\eta_1{\phi}(\xi-\eta-2^{k}e_1)
{\phi}(\eta+2^{k}e_1))\dd\eta\dd \xi\\
\le&\frac{C(\ln N)^2}{N}\sum_{N\le k\le(1+\delta)N}2^{-2k}\le\frac{C(\ln N)^2}{N2^{2N}}.
\end{align*}
Now we need to estimate the last two term on the right-hand side of inequality \eqref{U1}. At first, it holds that
\begin{align*}
&\|\partial_{x_1}U_0\|_{L^\infty_{t_0}L^\infty}\le C2^{\frac{3(1+\delta)}{2}N},\qquad \|\partial_{x_2}U_0\|_{L^\infty_{t_0}L^\infty}\le C2^{\frac{(1+\delta)}{2}N},\\
&\|\partial_{x_1x_2}U_0\|_{L^\infty_{t_0}L^\infty}\le C2^{\frac{3(1+\delta)}{2}N},\,\,\, \|\partial^2_{x_2}U_0\|_{L^\infty_{t_0}L^\infty}\le C2^{\frac{(1+\delta)}{2}N}.
\end{align*}
Based on these, we obtain that
\begin{align*}
&\Big\|\dot\Delta_0\int_0^{t_0} e^{({t_0}-s)\Delta}(U^{(2)}_0\partial_{x_2}U^{(2)}_0+\int_0^s\partial_{x_1x_2}U^{(2)}_0\dd \tau\partial_{x_1}U^{(2)}_0(s))\dd s\Big\|_{L^\infty}\\
\le& Ct_0\|U_0\|^2_{L^\infty_{t_0}L^\infty}+t^2_0\|\partial_{x_1x_2}U_0\|_{L^\infty_{t_0}L^\infty}
\|\partial_{x_1} U^2_0\|_{L^\infty_{t_0}L^\infty}\le C 2^{(3\delta-1)N},
\end{align*}
and
\begin{align*}
&\Big\|\dot\Delta_0\int_0^{t_0} e^{({t_0}-s)\Delta}(\int_0^s\partial_{x_2}(\partial_{x_1}U^{(1)}_0+\partial_{x_2}U^{(2)}_0)\dd \tau\partial_{x_2}U^{(2)}_0(s))\dd s\Big\|_{L^\infty}\\
\le&Ct^2_0(\|\partial_{x_1x_2}U^{(1)}_0\|_{L^\infty_{t_0}L^\infty}
+\|\partial^2_{x_2}U^{(2)}_0\|_{L^\infty_{t_0}L^\infty})\|\partial_{x_2}U^{(2)}_0\|_{L^\infty_{t_0}L^\infty}
\le C 2^{(2\delta-2)N}.
\end{align*}
Therefore, we have
\begin{align*}
\|U_1\|_{\dot B^{-\frac{1}{2}}_{4,q}}\ge C\delta(\ln N)^2-\frac{C(\ln N)^3}{N}-\frac{C(\ln N)^2}{N2^{2N}}
-C 2^{(3\delta-1)N}\ge  C\delta(\ln N)^2.
\end{align*}
Following the same method in \eqref{U2},  then we obtain
\begin{align*}
\|U_2(t_0)\|_{\dot B^{-\frac{1}{2}}_{4,q}}\le C\|U_2(t_0)\|_{\dot B^0_{[2,\infty],1}}\le C2^{(-\frac{1}{2}+2\e+\frac{3}{2}\delta)N}.
\end{align*}
Therefore, we conclude that, for $3\e+5\delta<1$, and $q>2$
\begin{align*}
\|u(t_0)\|_{\dot B^{-\frac{1}{2}}_{4,q}}&\ge \|U_1(t_0)\|_{\dot B^{-\frac{1}{2}}_{4,q}}-\|U_0(t_0)\|_{\dot B^{-\frac{1}{2}}_{4,q}}-\|U_2(t_0)\|_{\dot B^{-\frac{1}{2}}_{4,q}}\\
&\ge C\delta(\ln N)^2-C(\ln N)N^{\frac{1}{q}-\frac{1}{2}}-C2^{(-\frac{1}{2}+2\e+\frac{3}{2}\delta)N}\\
&\ge C\delta(\ln N)^2,
\end{align*}
associated with initial data $\|u_0\|_{\dot B^{-\frac{1}{2}}_{4,q}}\le C(\ln N)N^{\frac{1}{q}-\frac{1}{2}}$.

\section{Appendix}
\begin{proposition}Let $B(f,g)$ be defined by \eqref{B}. There exists an absolute constant $C$ such that
\[\sup_{t>0}\|B(f,g)(t)\|_{\dot B^{-\frac{1}{2}}_{4,2}}\le C\|f\|_{{\dot B^{-\frac{1}{2}}_{4,2}}}
\|g\|_{{\dot B^{-\frac{1}{2}}_{4,2}}}.\]
\end{proposition}
\begin{proof}
By Bony's paraproduct decomposition, we have
\begin{align*}
B(f,g)&=\sum_{j\in\ZZ}B(\dot S_{j-1} f, \dot\Delta_j g)+\sum_{j\in\ZZ}B(\dot\Delta_j f, \dot S_{j-1} g)+
\sum_{j\in\ZZ}B(\dot \Delta_j f, \dot {\widetilde\Delta}_j g)\\
&:=I+II+III.
\end{align*}
For $I$, using semi-group estimates and Bernstein's inequality, it follows that
\begin{align*}
\|I\|_{\dot B^{-\frac{1}{2}}_{4,2}}
\le& \int_0^t (t-s)^{-\frac{1}{2}}\Big\{2^{-\frac{3}{2}k}\Big\|\sum_{|j-k|\le 5}\dot\Delta_k((\dot S_{j-1}e^{s\Delta}f^i)(\dot\Delta_j\partial_ie^{s\Delta}g))\Big\|_{L^4}\Big\}_{\ell^2(k\in\ZZ)}\dd s\\
&+\int_0^t (t-s)^{-\frac{1}{2}}\Big\{2^{-\frac{3}{2}k}\Big\|\sum_{|j-k|\le 5}\dot\Delta_k((\int_0^s\dot S_{j-1}e^{\tau\Delta}\partial_i\partial_mf^m\dd\tau)(\dot\Delta_j\partial_ie^{s\Delta}g))\Big\|_{L^4}\Big\}_{\ell^2(k\in\ZZ)}\dd s\\
\le& \int_0^t (t-s)^{-\frac{1}{2}}\Big\{\sum_{|j-k|\le 5}2^{-\frac{3}{2}(k-j)}2^{-j}\big\|\dot S_{j-1}e^{s\Delta}f^i\|_{L^\infty}2^{-\frac{1}{2}j}\|\dot\Delta_j\partial_ie^{s\Delta}g)\big\|_{L^4}\Big\}_{\ell^2(k\in\ZZ)}\dd s\\
&+\int_0^t (t-s)^{-\frac{1}{2}}\Big\{\sum_{|j-k|\le 5}2^{-\frac{3}{2}(k-j)}2^{-\frac{3}{2}j}\int_0^s\big\|\dot S_{j-1}e^{\tau\Delta}\partial_i\Div f\big\|_{L^\infty}\dd\tau\big\|\dot\Delta_j\partial_ie^{s\Delta}g\big\|_{L^4}\Big\}_{\ell^2(k\in\ZZ)}\dd s\\
\le& C\int_0^t (t-s)^{-\frac{1}{2}}\Big(\big\|f\|_{\dot B^{-1}_{\infty,2}}\|\partial_ie^{s\Delta}g\big\|_{\dot B^{-\frac{1}{2}}_{4,2}}+\int_0^s\|e^{\tau\Delta}\partial_i\Div f\|_{\dot B^{-\frac{3}{2}}_{\infty, 2}}\dd \tau\|\nabla e^{s\Delta}g\|_{\dot B^0_{4,2}}\Big)\dd s\\
\le&C\int_0^t (t-s)^{-\frac{1}{2}}\Big(\big\|f\|_{\dot B^{-\frac{1}{2}}_{4,2}}s^{-\frac{1}{2}}\|g\big\|_{\dot B^{-\frac{1}{2}}_{4,2}}+\int_0^s\|e^{\tau\Delta}\partial_i\Div f\|_{\dot B^{-1}_{4, 2}}\dd \tau\|\nabla e^{s\Delta}g\|_{\dot B^0_{4,2}}\Big)\dd s\\
\le&C\int_0^t (t-s)^{-\frac{1}{2}}\big(s^{-\frac{1}{2}}+s^{-\frac{3}{4}}\int_0^s\tau^{-\frac{3}{4}}\dd \tau \big)\dd s\| f\|_{\dot B^{-\frac{1}{2}}_{4, 2}}\|g\|_{\dot B^{-\frac{1}{2}}_{4,2}}\\
\le&C\int_0^t (t-s)^{-\frac{1}{2}}s^{-\frac{1}{2}}\dd s\| f\|_{\dot B^{-\frac{1}{2}}_{4, 2}}\|g\|_{\dot B^{-\frac{1}{2}}_{4,2}}\le C\| f\|_{\dot B^{-\frac{1}{2}}_{4, 2}}\|g\|_{\dot B^{-\frac{1}{2}}_{4,2}}.
\end{align*}
For $II$, owning to $\dot B^0_{4,2}(\RR^2)\hookrightarrow L^4(\RR^2)$, we have
\begin{align*}
\|II\|_{\dot B^{-\frac{1}{2}}_{4,2}}
\le& \int_0^t (t-s)^{-\frac{1}{2}}\Big\{2^{-\frac{3}{2}k}\Big\|\sum_{|j-k|\le 5}\dot\Delta_k((\dot S_{j-1}\partial_ie^{s\Delta}g)(\dot \Delta_je^{s\Delta}f^i))\Big\|_{L^4}\Big\}_{\ell^2(k\in\ZZ)}\dd s\\
+\int_0^t& (t-s)^{-\frac{1}{2}}\Big\{2^{-\frac{3}{2}k}\Big\|\sum_{|j-k|\le 5}\dot\Delta_k((\dot S_{j-1}\partial_ie^{s\Delta}g)(\int_0^s\dot\Delta_je^{\tau\Delta}\partial_i\partial_mf^m\dd\tau))\Big\|_{L^4}\Big\}_{\ell^2(k\in\ZZ)}\dd s\\
\le \int_0^t &(t-s)^{-\frac{1}{2}}\Big\{\sum_{|j-k|\le 5}2^{-\frac{3}{2}(k-j)}2^{-\frac{1}{2}j}\|\dot S_{j-1}\partial_ie^{s\Delta}g)\big\|_{L^4}2^{-j}\big\|\dot \Delta_je^{s\Delta}f^i\|_{L^\infty}\Big\}_{\ell^2(k\in\ZZ)}\dd s\\
+\int_0^t &(t-s)^{-\frac{1}{2}}\Big\{\sum_{|j-k|\le 5}2^{-\frac{3}{2}(k-j)}2^{-\frac{3}{2}j}\int_0^s\big\|\dot \Delta_je^{\tau\Delta}\partial_i\Div f\big\|_{L^\infty}\dd\tau\big\|\dot S_{j-1}\partial_ie^{s\Delta}g\big\|_{L^4}\Big\}_{\ell^2(k\in\ZZ)}\dd s\\
\le C\int_0^t &(t-s)^{-\frac{1}{2}}\Big(\big\|f\|_{\dot B^{-1}_{\infty,2}}\|\partial_ie^{s\Delta}g\big\|_{\dot B^{-\frac{1}{2}}_{4,2}}+\int_0^s\|e^{\tau\Delta}\partial_i\Div f\|_{\dot B^{-\frac{3}{2}}_{\infty, 2}}\dd \tau\|\nabla e^{s\Delta}g\|_{\dot B^0_{4,2}}\Big)\dd s\\
\le C\int_0^t &(t-s)^{-\frac{1}{2}}s^{-\frac{1}{2}}\dd s\| f\|_{\dot B^{-\frac{1}{2}}_{4, 2}}\|g\|_{\dot B^{-\frac{1}{2}}_{4,2}}\le C\| f\|_{\dot B^{-\frac{1}{2}}_{4, 2}}\|g\|_{\dot B^{-\frac{1}{2}}_{4,2}}.
\end{align*}
For $III$, using Littlewood-Paley decomposition yields that
\begin{align*}
III=&\sum_{m\in\ZZ}\sum_{j\ge m-3}\dot \Delta_m\int_0^te^{(t-s)\Delta}\big(\dot\Delta_j e^{s\Delta}f\dot{\widetilde{\Delta}}_j\nabla e^{s\Delta}g+\dot\Delta_j \int_0^s\nabla\Div e^{\tau\Delta }f\dot{\widetilde{\Delta}}_j\nabla e^{s\Delta}g\big)\dd s\\
=&\sum_{m\in\ZZ}\sum_{m-3\le j< m+N_0}\dot \Delta_m\int_0^te^{(t-s)\Delta}\big(\dot\Delta_j e^{s\Delta}f\dot{\widetilde{\Delta}}_j\nabla e^{s\Delta}g+\dot\Delta_j \int_0^s\nabla\Div e^{\tau\Delta }f\dot{\widetilde{\Delta}}_j\nabla e^{s\Delta}g\big)\dd s\\
&+\sum_{m\in\ZZ}\sum_{j\ge m+N_0}\dot \Delta_m\int_0^te^{(t-s)\Delta}\big(\dot\Delta_j e^{s\Delta}f\dot{\widetilde{\Delta}}_j\nabla e^{s\Delta}g+\dot\Delta_j \int_0^s\nabla\Div e^{\tau\Delta }f\dot{\widetilde{\Delta}}_j\nabla e^{s\Delta}g\big)\dd s\\
:=&III_1+III_2.
\end{align*}
For $III_1$, thanks to Young's inequality, we obtain that
\begin{align*}
\|III_1\|_{\dot B^{-\frac{1}{2}}_{4,2}}\le&\Big\{\sum_{|k-m|\le 1}2^{-\frac{3}{2}k}\sum_{|j-m|\le N_0}\int_0^t(t-s)^{-\frac{1}{2}}\big\|\dot\Delta_k\dot\Delta_m(( \dot\Delta_je^{s\Delta}f^i)(\dot{\widetilde\Delta}_j\partial_ie^{s\Delta}g))\big\|_{L^4}\Big\}_{\ell^2(k)}\\
+\Big\{\sum_{|k-m|\le 1}&2^{-\frac{3}{2}k}\sum_{|j-m|\le N_0}\int_0^t(t-s)^{-\frac{1}{2}}\big\|\dot\Delta_k\dot\Delta_m((\int_0^s\dot \Delta_je^{\tau\Delta}\partial_i\partial_mf^m\dd\tau)(\dot{\widetilde\Delta}_j\partial_ie^{s\Delta}g))\big\|_{L^4}\Big\}_{\ell^2(k)}\\
\le C\Big\{2^{-\frac{3}{2}m}&\sum_{|j-m|\le N_0}\int_0^t(t-s)^{-\frac{1}{2}}\big\|\dot\Delta_m(( \dot\Delta_je^{s\Delta}f^i)(\dot{\widetilde\Delta}_j\partial_ie^{s\Delta}g))\big\|_{L^4}\Big\}_{\ell^2(m)}\\
+C\Big\{2^{-\frac{3}{2}m}&\sum_{|j-m|\le N_0}\int_0^t(t-s)^{-\frac{1}{2}}\big\|\dot\Delta_m((\int_0^s\dot \Delta_je^{\tau\Delta}\partial_i\partial_mf^m\dd\tau)(\dot{\widetilde\Delta}_j\partial_ie^{s\Delta}g))\big\|_{L^4}\Big\}_{\ell^2(m)}\\
\le C\Big\{\sum_{|j-m|\le N_0}&2^{-\frac{3}{2}(m-j)}\int_0^t(t-s)^{-\frac{1}{2}}2^{-j}\big\| \dot\Delta_je^{s\Delta}f\|_{L^\infty}2^{-\frac{1}{2}j}\|\dot{\widetilde\Delta}_j\partial_ie^{s\Delta}g\big\|_{L^4}\Big\}_{\ell^2(m)}\\
+C\Big\{\sum_{|j-m|\le N_0}&2^{-\frac{3}{2}(m-j)}\int_0^t(t-s)^{-\frac{1}{2}}\int_0^s2^{-\frac{3}{2}j}\big\|\dot \Delta_je^{\tau\Delta}\partial_i\partial_mf^m\|_{L^\infty}\dd\tau\|\dot{\widetilde\Delta}_j\partial_ie^{s\Delta}g\big\|_{L^4}\Big\}_{\ell^2(m)}\\
\le C\int_0^t  (t-s&)^{-\frac{1}{2}}\Big(\big\|f\|_{\dot B^{-1}_{\infty,2}}\|\partial_ie^{s\Delta}g\big\|_{\dot B^{-\frac{1}{2}}_{4,2}}+\int_0^s\|e^{\tau\Delta}\partial_i\Div f\|_{\dot B^{-\frac{3}{2}}_{\infty, 2}}\dd \tau\|\nabla e^{s\Delta}g\|_{\dot B^0_{4,2}}\Big)\dd s\\
\le C\int_0^t (t-s&)^{-\frac{1}{2}}s^{-\frac{1}{2}}\dd s\| f\|_{\dot B^{-\frac{1}{2}}_{4, 2}}\|g\|_{\dot B^{-\frac{1}{2}}_{4,2}}\le C\| f\|_{\dot B^{-\frac{1}{2}}_{4, 2}}\|g\|_{\dot B^{-\frac{1}{2}}_{4,2}}.
\end{align*}
For $III_2$, due to $L^2(\RR^2)\hookrightarrow \dot B^0_{2,2}(\RR^2)\hookrightarrow \dot B^{-\frac{1}{2}}_{4,2}(\RR^2)$ and Lemma \ref{multi}, it yields that
\begin{align*}
&\|III_2\|_{\dot B^{-\frac{1}{2}}_{4,2}}\le\|III_2\|_{L^2}\\
\le& \sum_{j\in\ZZ}\|\sum_{m\le j-N_0}\dot \Delta_m\int_0^te^{(t-s)\Delta}\big(\dot\Delta_j e^{s\Delta}f\dot{\widetilde{\Delta}}_j\nabla e^{s\Delta}g+\dot\Delta_j \int_0^s\nabla\Div e^{\tau\Delta }f\dot{\widetilde{\Delta}}_j\nabla e^{s\Delta}g\big)\dd s\|_{L^2}\\
\le&\sum_{j\in\ZZ}\Big\|\mathscr{F}^{-1}\Big\{\int_{\RR^2}\sum_{m\le j-N_0}\widehat{\varphi}_m(\xi)
\frac{e^{-t|\xi|^2}-e^{-t(|\xi-\eta|^2+|\eta|^2)}}{|\xi-\eta|^2+|\eta|^2-|\xi|^2}
\widehat{\varphi}_j(\xi-\eta)\widehat{f}(\xi-\eta)\widehat{\widetilde{\varphi}}_j(\eta)i\eta\widehat{g}(\eta)\dd\eta\Big\}\Big\|_{L^2}\\
&+\sum_{j\in\ZZ}\Big\|\mathscr{F}^{-1}\Big\{\int_{\RR^2}\sum_{m\le j-N_0}\widehat{\varphi}_m(\xi)\big(\frac{e^{-t|\xi|^2}-e^{-t|\eta|^2}}{|\eta|^2-|\xi|^2}-\frac{e^{-t|\xi|^2}-e^{-t(|\eta|^2+|\xi-\eta|^2)}}{|\eta|^2+|\xi-\eta|^2-|\xi|^2})\\
&\qquad\qquad\qquad\qquad\times\frac{(\xi-\eta)_k(\xi-\eta)_{\ell}}{|\xi-\eta|^2}\widehat{\varphi}_j(\xi-\eta)\widehat{f}^{\ell}(\xi-\eta)\widehat{\widetilde{\varphi}}_j(\eta)i\eta_k\widehat{g}(\eta)\dd\eta\Big\}\Big\|_{L^2}\\
\le& C\sum_{j\in\ZZ}2^{-2j}\|\dot\Delta_j f\|_{L^4}\|\dot{\widetilde{\Delta}}_j \nabla g\|_{L^4}
\le C\sum_{j\in\ZZ}\sum_{|j'-j|\le 1}2^{-2j}2^{j'}\|\dot\Delta_j f\|_{L^4}\|\dot{{\Delta}}_{j'} g\|_{L^4}\\
=&C\sum_{j\in\ZZ}\sum_{|j'-j|\le 1}2^{-\frac{3}{2}(j-j')}2^{-\frac{1}{2}j}\|\dot\Delta_j f\|_{L^4}2^{-\frac{1}{2}j'}\|\dot{{\Delta}}_{j'} g\|_{L^4}\le C\|f\|_{\dot B^{-\frac{1}{2}}_{4,2}}\|g\|_{\dot B^{-\frac{1}{2}}_{4,2}}.
\end{align*}
\end{proof}

\end{document}